\documentclass{elsarticle}
\usepackage{amsthm,amsmath,amsfonts,latexsym,amssymb,mathrsfs}
\usepackage{graphicx}
\usepackage{pstricks}
\usepackage{amsmath}
\usepackage{amsxtra}

\theoremstyle{plain}
\newtheorem{theorem}{Theorem}[section]
\newtheorem{lemma}[theorem]{Lemma}

\newtheorem{proposition}[theorem]{Proposition}
\newtheorem{corollary}[theorem]{Corollary}
\newtheorem{conjecture}[theorem]{Conjecture}

\newtheorem{problem}[theorem]{Problem}

\theoremstyle{definition}
\newtheorem{definition}[theorem]{Definition}

\newtheorem{example}[theorem]{Example}
\theoremstyle{remark}
\newtheorem{remark}[theorem]{Remark}

\numberwithin{equation}{section}

\begin{document}


\title{Aluthge transforms of $2$-variable weighted shifts}
\author{Ra\'{u}l E. Curto\footnote{The first-named author was partially supported by NSF Grants
DMS-0801168 and DMS-1302666.}}
\address{Department of Mathematics, The University of Iowa, Iowa City, Iowa
52242}
\ead{raul-curto@uiowa.edu}
\author{Jasang Yoon}
\address{School of Mathematical and Statistical Sciences, The University of Texas Rio Grande Valley,
Edinburg, Texas 78539}
\ead{jasang.yoon@utrgv.edu}

\date{}

\begin{abstract} \ We introduce two natural notions of multivariable Aluthge transforms (toral and spherical), and study their basic properties. \ In the case of $2$-variable
weighted shifts, we first prove that the toral Aluthge transform does not preserve (joint) hyponormality, in sharp contrast with the $1$-variable
case. \ Second, we identify a large class of $2$-variable weighted shifts for which hyponormality is preserved under both transforms. \ Third, we
consider whether these Aluthge transforms are norm-continuous. \ Fourth, we
study how the Taylor and Taylor essential spectra of $2$-variable
weighted shifts behave under the toral and spherical Aluthge transforms; as a special case, we consider the Aluthge transforms of the Drury-Arveson $2$-shift. \ Finally, we briefly discuss the class of spherically quasinormal $2$-variable weighted shifts, which are the fixed points for the spherical Aluthge transform.
\end{abstract}


\maketitle

\tableofcontents

\setcounter{tocdepth}{2}


\section{\label{Int}Introduction}

Over the last two decades, the Aluthge transform for bounded operators on Hilbert space has attracted considerable attention. \ In this note, we set out to extend the Aluthge transform to commuting $n$-tuples of bounded operators. \ We identify two natural notions (toral and spherical) and study their basic properties. \ We then focus on $2$-variable weighted shifts, for which much can be said. \

Let $\mathcal{H}$ be a complex Hilbert space and let $\mathcal{B}(\mathcal{H}%
)$ denote the algebra of bounded linear operators on $\mathcal{H}$. $\ $For $%
T\in \mathcal{B}(\mathcal{H})$, the polar decomposition of $T$ is $T \equiv V|T|$,
where $V$ is a partial isometry with $\textrm{ker} \; V=\textrm{ker} T$ and $|T|:=\sqrt{T^{\ast }T}$. \ The \textit{%
Aluthge transform} of $T$ is the operator $\widetilde{T}:=|T|^{\frac{1}{2}%
}V|T|^{\frac{1}{2}}$ \cite{Alu}. \ This transform was first considered by A. Aluthge, in
an effort to study $p$-hyponormal and $\mathrm{\log }$-hyponormal operators. \
Roughly speaking, the idea behind the Aluthge transform is to convert an
operator into another operator which shares with the first one many spectral
properties, but which is closer to being a normal operator. \

In recent years,
the Aluthge transform has received substantial attention. \ I.B. Jung, E. Ko and C. Pearcy
proved in \cite{JKP} that $T$ has a nontrivial invariant subspace if and
only if $\widetilde{T}$ does. \ (Since every normal operator has nontrivial
invariant subspaces, the Aluthge transform has a natural connection with the
invariant subspace problem.) \ In \cite{JKP2}, I.B. Jung, E. Ko and C. Pearcy also
proved that $T$ and $\widetilde{T}$ have the same spectrum. \ Moreover, in \cite{KiKo} (resp. \cite{Kim}) M.K. Kim and E. Ko (resp. F. Kimura) proved that $T$ has property $(\beta )$ if and only if $\widetilde{T}$ has property $(\beta )$. \ Finally, T. Ando proved in \cite{Ando} that $\left\| (T-\lambda )^{-1}\right\| \geq \left\| (\widetilde{T}-\lambda)^{-1}\right\| \;(\lambda \notin \sigma (T))$. \ (For additional results, see \cite{CJL} and \cite{Yam}.)

For a unilateral weighted shift $W_{\alpha }\equiv \mathrm{shift}(\alpha _{0},\alpha
_{1},\cdots )$, the Aluthge transform $\widetilde{W}_{\alpha }$ is also a unilateral weighted shift, given by
\begin{equation}
\widetilde{W}_{\alpha }\equiv \mathrm{shift}(\sqrt{\alpha _{0}\alpha _{1}},%
\sqrt{\alpha _{1}\alpha _{2}},\cdots )\text{ (see \cite{LLY}).}
\label{Alu-hypo}
\end{equation}%
It is easy to see that $W_{\alpha }$ is
hyponormal if and only if $\alpha _{0}\leq \alpha _{1}\leq \cdots $. \ Thus,
by (\ref{Alu-hypo}), if $W_{\alpha }$ is hyponormal, then the Aluthge
transform $\widetilde{W}_{\alpha }$ of $W_{\alpha }$ is also hyponormal. \
However, the converse is not true in general. \ For example, if $W_{\alpha
}\equiv \mathrm{shift}\left( \frac{1}{2},2,\frac{1}{2},2,\frac{1}{2}%
,2,\cdots \right) $, then $W_{\alpha }$ is clearly not hyponormal but the
Aluthge transform $\widetilde{W}_{\alpha } \equiv U_{+}$ is subnormal. \ (Here and in what follows, $U_{+}$ denotes the (unweighted) unilateral shift.) \ In \cite%
{LLY}, S.H. Lee, W.Y. Lee and the second-named author showed that
for $k\geq 2$, the Aluthge transform, when acting on weighted shifts, need not preserve
$k$-hyponormality. \ Finally, G. Exner proved in \cite{Ex} that the Aluthge transform of a subnormal weighted shift need not be subnormal.

In this article, we introduce two Aluthge transforms of commuting pairs of Hilbert space operators, with special emphasis on $2$-variable
weighted shifts $W_{(\alpha ,\beta )}\equiv (T_{1},T_{2})$. \ Since a priori there are several possible notions, we discuss two plausible definitions and their basic properties in Sections \ref{Sect3} and \ref{Sec2}. \ Our research will allow us to compare both definitions in terms of how well they generalize the $1$-variable notion. \ After discussing some basic properties of each Aluthge transform, we proceed to study both transforms in the case of $2$-variable weighted shifts. \ We consider topics such as preservation of joint hyponormality, norm continuity, and Taylor spectral behavior.

For $i=1,2$, we consider the polar decomposition $T_{i}\equiv V_{i}\left\vert
T_{i}\right\vert $, and we let
\begin{equation}
\widetilde{T_i}:=|T_{i}|^{\frac{1}{2}}V_{i}|T_{i}|^{\frac{1}{2}} \; \; \; (i=1,2) \label{Def-Alu}
\end{equation}
denote the classical ($1$-variable) Aluthge transform. \ The {\it toral} Aluthge transform of the pair $(T_{1},T_{2})$ is $\widetilde{(T_{1},T_{2})}:=(\widetilde{T_1},\widetilde{T_2})$. \ For a $2$-variable weighted shift $W_{(\alpha ,\beta )} \equiv (T_{1},T_{2})$, we denote the toral Aluthge transform
of $W_{(\alpha ,\beta )}$ by $\widetilde{W}_{(\alpha ,\beta )}$. \

As we will see in Proposition \ref{commuting1}, the commutativity of $\widetilde{W}_{(\alpha ,\beta )}$ does not automatically follow from the commutativity of $W_{(\alpha ,\beta )}$; in fact, the necessary and sufficient condition to preserve commutativity is
\begin{equation}
\alpha _{(k_{1},k_{2}+1)}\alpha
_{(k_{1}+1,k_{2}+1)}=\alpha _{(k_{1}+1,k_{2})}\alpha _{(k_{1},k_{2}+2)} \; \; (\text{for all }k_{1},k_{2}\geq 0). \label{alphacomm}
\end{equation}

Under this assumption, and in sharp contrast with the $1$-variable situation, it is possible to exhibit a {\it commuting subnormal} pair $W_{(\alpha ,\beta )}$ such that $\widetilde{W}_{(\alpha ,\beta )}$ is commuting and {\it not hyponormal}. \ As a matter of fact, in Theorem \ref{example100} we construct a class of subnormal $2$-variable weighted shifts $W_{(\alpha ,\beta )}$ whose cores are of tensor form, and for which the hyponormality of $\widetilde{W}_{(\alpha ,\beta )}$ can be described entirely by two parameters. \ As a result, we obtain a rather large class of subnormal $2$-variable weighted shifts with non-hyponormal toral Aluthge transforms.

There is a second plausible definition of Aluthge transform, which uses a joint polar decomposition. \ Assume that we have a decomposition of the form
\begin{equation*}
(T_{1},T_{2})\equiv\left( V_{1}P,V_{2}P\right) \text{,}
\end{equation*}
where $P:=\sqrt{T_{1}^{\ast }T_{1}+T_{2}^{\ast }T_{2}}$. \ Now, let
\begin{equation}
\widehat{(T_1,T_2)}:=\left( \sqrt{P}V_{1}\sqrt{P},\sqrt{P}V_{2}%
\sqrt{P}\right) \text{,}  \label{Def-Alu1}
\end{equation}%
We refer to $\widehat{\mathbf{T}}$ as the {\it spherical} Aluthge transform of $\mathbf{T}$. \
Even though $%
\widehat{T}_{1}=\sqrt{P}V_{1}\sqrt{P}$ is not the Aluthge transform of $%
T_{1} $, we observe in Section \ref{Sec2} that $Q:=\sqrt{V_{1}^{\ast }V_{1}+V_{2}^{\ast }V_{2}}$ is a
(joint) partial isometry; for, $PQ^2P=P^2$, from which it follows that $Q$ is isometric on the range of $P$. \ We will prove in Section \ref{Sec2} that this particular definition of the Aluthge transform preserves commutativity. \

There is another useful aspect of the spherical Aluthge transform, which we now mention. \ If we consider the fixed points of this transform acting on $2$-variable weighted shifts, then we obtain an appropriate generalization of the concept of quasinormality. \ Recall that a Hilbert space operator $T$ is said to be {\it quasinormal} if $T$ commutes with the positive factor $P$ in the polar decomposition $T\equiv VP$; equivalently, if $V$ commutes with $P$. \ It follows easily that $T$ is quasinormal if and only if $T=\widetilde{T}$, that is, if and only if $T$ is a fixed point for the Aluthge transform. \ In Section \ref{Spherquasi}, we prove that if a $2$-variable weighted shift $W_{(\alpha,\beta)}=(T_1,T_2)$ satisfies $W_{(\alpha,\beta)}=\left(
\widehat{T}_{1},\widehat{T}_{2}\right)$, then $T_1^*T_1+T_2^*T_2$ is, up to scalar multiple, a spherical isometry. \ It follows that we can then study some properties of the spherical Aluthge transform using well known results about spherical isometries.

In this paper, we also focus on the following three basic problems.

\begin{problem}
\label{problem 1}Let $k\geq 1$ and assume that $W_{(\alpha ,\beta )}$ is $k$-hyponormal. \ Does it follow that the toral Aluthge transform $\widetilde{W}_{(\alpha ,\beta )}$ is $k$%
-hyponormal? \ What about the case of the spherical Aluthge transform $\widehat{W}_{(\alpha ,\beta )}$?
\end{problem}

\begin{problem}
\label{problem 2} Is the toral Aluthge transform $\left( S,T\right)
\rightarrow \left( \widetilde{S},\widetilde{T}\right) $ continuous in the uniform topology? \ Similarly, does continuity hold for the spherical Aluthge transform $\widehat{W}_{(\alpha ,\beta )}$?
\end{problem}

\begin{problem}
\label{problem 3}Does the Taylor spectrum (resp. Taylor essential spectrum) of $%
\widetilde{W}_{(\alpha ,\beta )}$ equal to that of $W_{(\alpha ,\beta )}$ ? \ What about the case of the spherical Aluthge transform $\widehat{W}_{(\alpha ,\beta )}$?
\end{problem}


\section{\label{Sect1}Notation and Preliminaries}

\subsection{Subnormality and $k$-hyponormality} \ We say that $T\in \mathcal{B}(\mathcal{H})$ is \textit{normal} if $T^{\ast
}T=TT^{\ast }$, \textit{quasinormal} if $T$ commutes with $T^{\ast }T$, \textit{subnormal} if $T=N|_{\mathcal{H}}$, where $N$ is
normal and $N(\mathcal{H}\mathcal{)}$ $\mathcal{\subseteq H}$, and \textit{%
hyponormal} if $T^{\ast }T\geq TT^{\ast }$. \ For $S,T\in \mathcal{B}(%
\mathcal{H})$, let $[S,T]:=ST-TS$. \ We say that an $n$-tuple $\mathbf{T} \equiv (T_{1},\cdots ,T_{n})$ of operators on $\mathcal{H}$ is (jointly) \textit{%
hyponormal} if the operator matrix
\begin{equation*}
\lbrack \mathbf{T}^{\ast },\mathbf{T]:=}\left(
\begin{array}{llll}
\lbrack T_{1}^{\ast },T_{1}] & [T_{2}^{\ast },T_{1}] & \cdots & [T_{n}^{\ast
},T_{1}] \\
\lbrack T_{1}^{\ast },T_{2}] & [T_{2}^{\ast },T_{2}] & \cdots & [T_{n}^{\ast
},T_{2}] \\
\text{ \thinspace \thinspace \quad }\vdots & \text{ \thinspace \thinspace
\quad }\vdots & \ddots & \text{ \thinspace \thinspace \quad }\vdots \\
\lbrack T_{1}^{\ast },T_{n}] & [T_{2}^{\ast },T_{n}] & \cdots & [T_{n}^{\ast
},T_{n}]%
\end{array}%
\right)
\end{equation*}%
is positive on the direct sum of $n$ copies of $\mathcal{H}$ (cf. \cite{Ath}%
, \cite{CMX}). \ For instance, if $n=2$,
\begin{equation*}
\lbrack \mathbf{T}^{\ast },\mathbf{T}\rbrack = \left(
\begin{array}{ll}
\lbrack T_{1}^{\ast },T_{1}] & [T_{2}^{\ast },T_{1}] \\
\lbrack T_{1}^{\ast },T_{2}] & [T_{2}^{\ast },T_{2}]%
\end{array}%
\right) \text{.}
\end{equation*}%
For $k\geq 1$ $T$ is $k$\textit{-hyponormal} if $\left( I,T,\cdots
,T^{k}\right) $ is (jointly) hyponormal. \ The Bram-Halmos characterization
of subnormality (\cite[III.1.9]{Con}) can be paraphrased as follows: \ $T$
is subnormal if and only if $T$ is $k$-hyponormal for every $k\geq 1$ (\cite[%
Proposition 1.9]{CMX}). \ The $n$-tuple $\mathbf{T}\equiv
(T_{1},T_{2},\cdots ,T_{n})$ is said to be \textit{normal} if $\mathbf{T}$
is commuting and each $T_{i}$ is normal, and $\mathbf{T}$ is \textit{%
subnormal} if $\mathbf{T}$ is the restriction of a normal $n$-tuple to a
common invariant subspace. \ In particular, a commuting pair $\mathbf{T}%
\equiv (T_{1},T_{2})$ is said to be $k$\textit{-hyponormal }$(k\geq 1)$ \cite%
{CLY1} if
\begin{equation*}
\mathbf{T}(k):=(T_{1},T_{2},T_{1}^{2},T_{2}T_{1},T_{2}^{2},\cdots
,T_{1}^{k},T_{2}T_{1}^{k-1},\cdots ,T_{2}^{k})
\end{equation*}%
is hyponormal, or equivalently
\begin{equation*}
\lbrack \mathbf{T}(k)^{\ast },\mathbf{T}(k)]=([(T_{2}^{q}T_{1}^{p})^{\ast
},T_{2}^{m}T_{1}^{n}])_{_{1\leq p+q\leq k}^{1\leq n+m\leq k}}\geq 0.
\end{equation*}%
Clearly, for $T \in \mathcal{B}(\mathcal{H})$ we have
$$
\textrm{normal} \Rightarrow  \textrm{quasinormal} \Rightarrow  \textrm{subnormal} \Rightarrow  k \textrm{-hyponormal} \Rightarrow  \textrm{hyponormal}.
$$
As one might expect, there is a version of the Bram-Halmos Theorem in several variables, proved in \cite{CLY1}: a commuting pair which is $k$-hyponormal for every $k \ge 1$ is necessarily subnormal.

\subsection{Unilateral weighted shifts} \ For $\alpha \equiv \{\alpha _{n}\}_{n=0}^{\infty }$ a bounded sequence of
positive real numbers (called \textit{weights}), let $W_{\alpha }:\ell ^{2}(%
\mathbb{Z}_{+})\rightarrow \ell ^{2}(\mathbb{Z}_{+})$ be the associated
\textit{unilateral weighted shift}, defined by $W_{\alpha }e_{n}:=\alpha
_{n}e_{n+1}\;($all $n\geq 0)$, where $\{e_{n}\}_{n=0}^{\infty }$ is the
canonical orthonormal basis in $\ell ^{2}(\mathbb{Z}_{+}).$ \ We will often
write $\mathrm{shift}(\alpha _{0},\alpha _{1},\cdots )$ to denote the
weighted shift $W_{\alpha }$ with a weight sequence $\{\alpha
_{n}\}_{n=0}^{\infty }$. \ As usual, the (unweighted) unilateral shift will be denoted by $U_{+}:=\mathrm{shift} (1,1,\cdots)$. \ The \textit{moments} of $W_{\alpha }$ are given
by
\begin{equation*}
\gamma _{k}\equiv \gamma _{k}(W_{\alpha }):=\left\{
\begin{array}{cc}
1 & \text{if }k=0\text{ } \\
\alpha _{0}^{2}\alpha _{1}^{2}\cdots \alpha _{k-1}^{2} & \text{if }k>0.%
\end{array}%
\right.
\end{equation*}%

\subsection{$2$-variable weighted shifts} \ Similarly, consider double-indexed positive bounded sequences $\alpha _{%
\mathbf{k}},\beta _{\mathbf{k}}\in \ell ^{\infty }(\mathbb{Z}_{+}^{2})$, $%
\mathbf{k}\equiv (k_{1},k_{2})\in \mathbb{Z}_{+}^{2}:=\mathbb{Z}_{+}\times
\mathbb{Z}_{+}$ and let $\ell ^{2}(\mathbb{Z}_{+}^{2})$\ be the Hilbert
space of square-summable complex sequences indexed by $\mathbb{Z}_{+}^{2}$.

We define the $2$-variable weighted shift $W_{(\alpha ,\beta )}\equiv
(T_{1},T_{2})$\ by
\begin{equation*}
T_{1}e_{\mathbf{k}}:=\alpha _{\mathbf{k}}e_{\mathbf{k+}\varepsilon _{1}}%
\text{ and }T_{2}e_{\mathbf{k}}:=\beta _{\mathbf{k}}e_{\mathbf{k+}%
\varepsilon _{2}},
\end{equation*}%
where $\mathbf{\varepsilon }_{1}:=(1,0)$ and $\mathbf{\varepsilon }%
_{2}:=(0,1)$ (see Figure \ref{Figure 1}(i)). \ Clearly,
\begin{equation}
T_{1}T_{2}=T_{2}T_{1}\Longleftrightarrow \alpha _{\mathbf{k+}\varepsilon _{2}}\beta _{%
\mathbf{k}}=\beta _{\mathbf{k+}\varepsilon
_{1}}\alpha _{\mathbf{k}} \;(\text{all }\mathbf{k}\in \mathbb{Z}_{+}^{2}).
\label{commuting}
\end{equation}%
In an entirely similar way, one can define multivariable weighted shifts. \
Trivially, a pair of unilateral weighted shifts $W_{\sigma }$ and $W_{\tau
} $ gives rise to a $2$-variable weighted shift $W_{(\alpha ,\beta )}\equiv
\mathbf{T}\equiv (T_{1},T_{2})$, if we let $\alpha _{(k_{1},k_{2})}:=\sigma
_{k_{1}}$ and $\beta _{(k_{1},k_{2})}:=\tau _{k_{2}}\;($all $k_{1},k_{2}\in
\mathbb{Z}_{+})$. \ In this case, $W_{(\alpha ,\beta )}$ is subnormal (resp.
hyponormal) if and only if $T_{1}$ and $T_{2}$ are as well; in fact, under
the canonical identification of $\ell ^{2}(\mathbb{Z}_{+}^{2})$ with $\ell
^{2}(\mathbb{Z}_{+})\bigotimes \ell ^{2}(\mathbb{Z}_{+})$, we have $%
T_{1}\cong I\bigotimes W_{\sigma }$ and $T_{2}\cong W_{\tau }\bigotimes I$,
and $W_{(\alpha ,\beta )}$ is also doubly commuting. \ For this reason, we
do not focus attention on shifts of this type, and use them only when the
above mentioned triviality is desirable or needed.

\setlength{\unitlength}{1mm} \psset{unit=1mm}
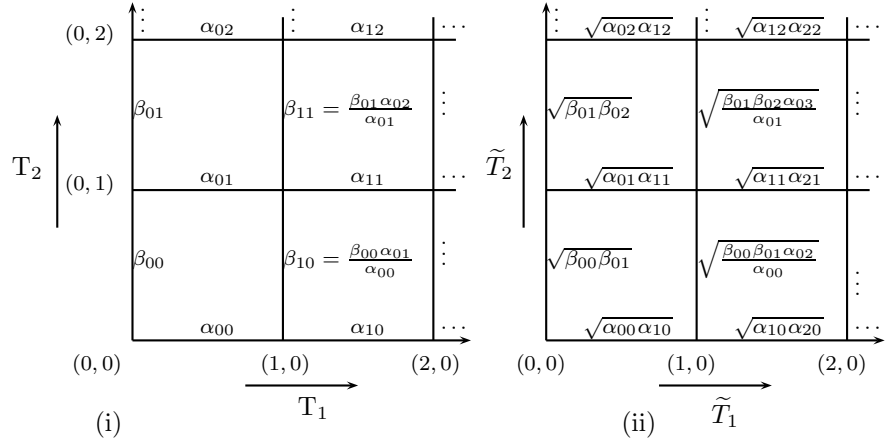
\begin{figure}[th]
\begin{center}
\begin{picture}(135,70)

\psline{->}(20,20)(65,20)
\psline(20,40)(63,40)
\psline(20,60)(63,60)
\psline{->}(20,20)(20,65)
\psline(40,20)(40,63)
\psline(60,20)(60,63)

\put(12,16){\footnotesize{$(0,0)$}}
\put(37,16){\footnotesize{$(1,0)$}}
\put(57,16){\footnotesize{$(2,0)$}}

\put(29,21){\footnotesize{$\alpha_{00}$}}
\put(49,21){\footnotesize{$\alpha_{10}$}}
\put(61,21){\footnotesize{$\cdots$}}

\put(29,41){\footnotesize{$\alpha_{01}$}}
\put(49,41){\footnotesize{$\alpha_{11}$}}
\put(61,41){\footnotesize{$\cdots$}}

\put(29,61){\footnotesize{$\alpha_{02}$}}
\put(49,61){\footnotesize{$\alpha_{12}$}}
\put(61,61){\footnotesize{$\cdots$}}

\psline{->}(35,14)(50,14)
\put(42,10){$\rm{T}_1$}
\psline{->}(10,35)(10,50)
\put(4,42){$\rm{T}_2$}

\put(11,40){\footnotesize{$(0,1)$}}
\put(11,60){\footnotesize{$(0,2)$}}

\put(20,30){\footnotesize{$\beta_{00}$}}
\put(20,50){\footnotesize{$\beta_{01}$}}
\put(21,61){\footnotesize{$\vdots$}}

\put(40,30){\footnotesize{$\beta_{10}=\frac{\beta_{00}\alpha_{01}}{\alpha_{00}}$}}
\put(40,50){\footnotesize{$\beta_{11}=\frac{\beta_{01}\alpha_{02}}{\alpha_{01}}$}}
\put(41,61){\footnotesize{$\vdots$}}

\put(61,30){\footnotesize{$\vdots$}}
\put(61,50){\footnotesize{$\vdots$}}

\put(15,8){(i)}


\put(85,8){(ii)}

\psline{->}(90,14)(105,14)
\put(97,9){$\widetilde{T}_{1}$}
\psline{->}(72,35)(72,50)
\put(67,42){$\widetilde{T}_{2}$}

\psline{->}(75,20)(120,20)
\psline(75,40)(118,40)
\psline(75,60)(118,60)

\psline{->}(75,20)(75,65)
\psline(95,20)(95,63)
\psline(115,20)(115,63)

\put(71,16){\footnotesize{$(0,0)$}}
\put(91,16){\footnotesize{$(1,0)$}}
\put(111,16){\footnotesize{$(2,0)$}}

\put(80,21){\footnotesize{$\sqrt{\alpha_{00}\alpha_{10}}$}}
\put(100,21){\footnotesize{$\sqrt{\alpha_{10}\alpha_{20}}$}}
\put(116,21){\footnotesize{$\cdots$}}

\put(80,41){\footnotesize{$\sqrt{\alpha_{01}\alpha_{11}}$}}
\put(100,41){\footnotesize{$\sqrt{\alpha_{11}\alpha_{21}}$}}
\put(116,41){\footnotesize{$\cdots$}}

\put(80,61){\footnotesize{$\sqrt{\alpha_{02}\alpha_{12}}$}}
\put(100,61){\footnotesize{$\sqrt{\alpha_{12}\alpha_{22}}$}}
\put(116,61){\footnotesize{$\cdots$}}

\put(75,30){\footnotesize{$\sqrt{\beta_{00}\beta_{01}}$}}
\put(75,50){\footnotesize{$\sqrt{\beta_{01}\beta_{02}}$}}
\put(76,61){\footnotesize{$\vdots$}}

\put(95,30){\footnotesize{$\sqrt{\frac{\beta_{00}\beta_{01}\alpha_{02}}{\alpha_{00}}}$}}
\put(95,50){\footnotesize{$\sqrt{\frac{\beta_{01}\beta_{02}\alpha_{03}}{\alpha_{01}}}$}}
\put(96,61){\footnotesize{$\vdots$}}

\put(116,26){\footnotesize{$\vdots$}}
\put(116,50){\footnotesize{$\vdots$}}

\end{picture}
\end{center}
\caption{Weight diagram of a commutative $2$-variable weighted shift $W_{(\protect\alpha ,%
\protect\beta )}\equiv (T_{1},T_{2})$ and weight diagram of its toral Aluthge
transform $\widetilde{W}_{(\protect\alpha ,\protect\beta )}\equiv \widetilde{%
(T_{1},T_{2})}\equiv (\widetilde{T}_{1},\widetilde{T}_{2})$, respectively. \ Observe that the commutativity of $\widetilde{W}_{(\protect\alpha ,\protect\beta )}$ requires (\ref{alphacomm}).}
\label{Figure 1}
\end{figure}

\subsection{Moments and subnormality} \ Given $\mathbf{k}\in \mathbb{Z}_{+}^{2}$, the \textit{moments} of $W_{(\alpha ,\beta
)}$ are
\begin{eqnarray*}
\gamma _{\mathbf{k}} &\equiv &\gamma _{\mathbf{k}}(W_{(\alpha ,\beta )}) \\
&:&=%
\begin{cases}
1, & \text{if }k_{1}=0\text{ and }k_{2}=0 \\
\alpha _{(0,0)}^{2}\cdots \alpha _{(k_{1}-1,0)}^{2}, & \text{if }k_{1}\geq 1%
\text{ and }k_{2}=0 \\
\beta _{(0,0)}^{2}\cdots \beta _{(0,k_{2}-1)}^{2}, & \text{if }k_{1}=0\text{
and }k_{2}\geq 1 \\
\alpha _{(0,0)}^{2}\cdots \alpha _{(k_{1}-1,0)}^{2}\beta
_{(k_{1},0)}^{2}\cdots \beta _{(k_{1},k_{2}-1)}^{2}, & \text{if }k_{1}\geq 1%
\text{ and }k_{2}\geq 1.%
\end{cases}%
\end{eqnarray*}
We remark that, due to the commutativity condition (\ref{commuting}), $%
\gamma _{\mathbf{k}}$ can be computed using any nondecreasing path from $%
(0,0)$ to $(k_{1},k_{2})$. \ 

We now recall a well known characterization of
subnormality for multivariable weighted shifts \cite{JeLu}, due to C. Berger
(cf. \cite[III.8.16]{Con}) and independently established by Gellar and
Wallen \cite{GeWa} in the $1$-variable case: $W_{(\alpha,\beta)} \equiv (T_{1},T_{2})$ admits a commuting normal extension if and only if there is a
probability measure $\mu $ (which we call the \textit{Berger measure} of $W_{(\alpha,\beta)}$ defined on the $2$-dimensional rectangle $R=[0,a_{1}]\times
\lbrack 0,a_{2}]$ (where $a_{i}:=\left\Vert T_{i}\right\Vert ^{2}$) such that%
\begin{equation*}
\gamma _{\mathbf{k}}(W_{(\alpha ,\beta )})=\int_{R}t^{\mathbf{k}}d\mu
(s,t):=\int_{R}s^{k_{1}}t^{k_{2}}d\mu (s,t)\text{, for all }\mathbf{k}\in
\mathbb{Z}_{+}^{2}.
\end{equation*}%
\ For $i\geq 1$, we let $\mathcal{L}_{i}:=\bigvee \{e_{k_{1}}:k_{1}\geq i\}$
denote the invariant subspace obtained by removing the first $i$ vectors in
the canonical orthonormal basis of $\ell ^{2}(\mathbb{Z}_{+})$; we also let $\mathcal{L}:=\mathcal{L}_{1}$. \ In the $1$%
-variable case, if $W_{\alpha }$ is subnormal with Berger measure $\xi
_{\alpha }$, then the Berger measure of $W_{\alpha }|_{\mathcal{L}_{i}}$ is $%
d\xi _{\alpha }|_{\mathcal{L}_{i}}(s):=\frac{s^{i}}{\gamma _{i}\left(
W_{\alpha }\right) }d\xi _{\alpha }(s)$, where $W_{\alpha }|_{\mathcal{L}%
_{i}}$ means the restriction of $W_{\alpha }$ to $\mathcal{L}_{i}$. \ As above, $U_{+}$ is the (unweighted) unilateral
shift, and for $0<a<1$ we let $S_{a}:=\mathrm{shift}(a,1,1,\cdots )$. \ Let $%
\delta _{p}$ denote the point-mass probability measure with support the
singleton set $\{p\}$. \ Observe that $U_{+}$ and $S_{a}$ are subnormal,
with Berger measures $\delta _{1}$ and $(1-a^{2})\delta _{0}+a^{2}\delta
_{1} $, respectively.

\subsection{Taylor spectra} \ We conclude this section with some terminology needed to describe the Taylor and Taylor essential spectra of commuting $n$-tuples of operators on a Hilbert space. \ Let $\Lambda \equiv \Lambda _{n}[e]$ be the \textit{complex exterior algebra
}on $n$ generators $e_{1},\ldots ,e_{n}$ with identity $e_{0}\equiv 1$,
multiplication denoted by $\wedge $ (wedge product) and complex
coefficients, subject to the collapsing property $e_{i}\wedge
e_{j}+e_{j}\wedge e_{i}=0$ $\left( 1\leq i,j\leq 1\right) $. \ If one
declares $\left\{ e_{I}\equiv e_{i_{1}}\wedge \ldots \wedge e_{i_{k}}:I\in
\{1,\ldots ,n\}\right\} $ to be an orthonormal basis, the exterior algebra
becomes a Hilbert space with the canonical inner product, i.e., $\left\langle
e_{I},e_{J}\right\rangle :=0$ if $I\neq J$, $\left\langle
e_{I},e_{J}\right\rangle :=1$ if $I=J$. \ It also admits an orthogonal
decomposition $\Lambda =\oplus _{i=0}^{n}\Lambda ^{i}$ with $\Lambda
^{i}\wedge \Lambda ^{k}\subset \Lambda ^{i+k}$. \ Moreover, $\dim \Lambda
^{k}=\binom{n}{k}=\frac{n!}{k!(n-k)!}$. \ Let $E_{i}:\Lambda \rightarrow
\Lambda $ denote the \textit{creation operator}, given by $\xi \longmapsto e_{i}\wedge \xi $ \; ($i=1,\cdots,n$). \ We recall that $E_{i}^{\ast
}E_{j}+E_{j}E_{i}^{\ast }=\delta _{ij}$ and $E_{i}$ is a partial isometry (all $i,j=1,\cdots,n$). \ Consider a Hilbert space $\mathcal{H}$ and set $\Lambda \left( \mathcal{H}\right) :=\oplus
_{i=0}^{n}\mathcal{H}\otimes _{\mathbb{C}}\Lambda ^{i}$.\ \ For a commuting $n$-tuple $\mathbf{T} \equiv \left(
T_{1},\ldots ,T_{n}\right)$ of bounded operators on $\mathcal{H}$, define
$$
D_{\mathbf{T}}:\Lambda \left( \mathcal{H}\right) \rightarrow \Lambda \left( \mathcal{H}\right)
\text{ by }D_{\mathbf{T}}\left( x\otimes \xi \right)
=\sum_{i=1}^{n}T_{i}x\otimes e_{i}\wedge \xi \text{.}
$$
Then $D_{\mathbf{T}}\circ D_{\mathbf{T%
}}=0$, so $Ran D_{\mathbf{T}}\subseteq Ker D_{\mathbf{T}}$. \ This
naturally leads to a cochain complex, called the \textit{Koszul complex} $K(%
\mathbf{T,}\mathcal{H})$ associated to $\mathbf{T}$ on $\mathcal{H}$, as follows:%
$$
K(\mathbf{T,}\mathcal{H}):0 \overset{0}{\rightarrow } \mathcal{H}\otimes
\wedge ^{0} \overset{D_{\mathbf{T}}^{0}}{\rightarrow } \mathcal{H}%
\otimes \wedge ^{1} \overset{D_{\mathbf{T}}^{1}}{\rightarrow } \cdots  
\overset{D_{\mathbf{T}}^{n-1}}{\rightarrow } \mathcal{H}\otimes \wedge ^{n}
\overset{D_{\mathbf{T}}^{n}\equiv 0}{\rightarrow } 0\text{,}
$$
where $D_{\mathbf{T}}^{i}$ denotes the restriction of $D_{\mathbf{T}}$ to
the subspace $\mathcal{H}\otimes \wedge ^{i}$. \ We define $\mathbf{T}$ to be \textit{%
invertible} in case its associated Koszul complex $K(\mathbf{T,}\mathcal{H})$ is
exact. \ Thus, we can define the Taylor spectrum $\sigma _{T}(\mathbf{T})$
of $\mathbf{T}$ as follows:\
\begin{equation*}
\begin{tabular}{l}
$\sigma _{T}(\mathbf{T})$ \\
$:=\left\{ (\lambda _{1},\cdots ,\lambda _{n})\in \mathbb{C}^{n}:K\left(
\left( T_{1}-\lambda _{1},\ldots ,T_{n}-\lambda _{n}\right) \mathbf{,}\text{
}\mathcal{H}\right) \text{ is not invertible}\right\} $.%
\end{tabular}%
\end{equation*}
\ J. L. Taylor showed that, if $\mathcal{H} \neq \{0\}$, then $\sigma _{T}(\mathbf{T})$ is a nonempty, compact subset of the
polydisc of multiradius $r(\mathbf{T}):=(r(T_{1}),\cdots ,r(T_{n})),$ where $%
r(T_{i})$ is the spectral radius of $T_{i}$ \; ($i=1,\cdots,n$) (\cite{Tay1}, \cite{Tay2}). \ For additional facts about this notion of joint spectrum, the reader is referred to \cite{Cu1}, \cite{Appl} and \cite{Cu3}.

\bigskip


\section{\label{Sect3}The Toral Aluthge Transform}

We will now gather several well known auxiliary results which are needed for the
proofs of the main results of this section. \ We begin with a criterion for the $k$%
-hyponormality $\left( k\geq 1\right) $ of $2$-variable weighted shifts. \ But first we need to describe concretely the toral Aluthge transform of a $2$-variable weighted shift, and the necessary and sufficient condition to guarantee its commutativity.

\begin{lemma} \label{CartAlu}
Let $W_{(\alpha ,\beta )} \equiv \left(T_{1},T_{2}\right)$ be a $2$-variable weighted shift. \ Then
$$
\widetilde{T}_{1}e_{\mathbf{k}}=\sqrt{\alpha_{\mathbf{k}}\alpha_{\mathbf{k}+\mathbf{\varepsilon}_{1}}}e_{\mathbf{k}+\mathbf{\varepsilon}_{1}}
$$
and
$$
\widetilde{T}_{2}e_{\mathbf{k}}=\sqrt{\beta_{\mathbf{k}}\beta_{\mathbf{k}+\mathbf{\varepsilon}_{2}}}e_{\mathbf{k}+\mathbf{\varepsilon}_{2}}
$$
for all $\mathbf{k} \in \mathbb{Z}_+^2$.
\end{lemma}

\begin{proof}
Straightforward from (\ref{Def-Alu}).
\end{proof}

In the following result we prove that the commutativity of $\left( \widetilde{T}_{1},%
\widetilde{T}_{2}\right)$ requires a condition on the weight sequences.

\subsection{Commutativity of the toral Aluthge transform}

\begin{proposition}
\label{commuting1} Let $W_{(\alpha ,\beta )}$ be a commuting $2$-variable weighted shift, with weight diagram given
by Figure \ref{Figure 1}(i). \ Then
\begin{eqnarray}
\widetilde{W}_{(\alpha ,\beta )} &\equiv &\left( \widetilde{T}_{1},%
\widetilde{T}_{2}\right) \text{ is commuting} \nonumber \\
&\Longleftrightarrow &\alpha _{\mathbf{k+}\varepsilon _{2}}\alpha _{\mathbf{k}+\varepsilon _{1}+\varepsilon _{2}}=\alpha_{\mathbf{k+}\varepsilon _{1}}\alpha _{\mathbf{k}+2\varepsilon _{2}}\label{prop1eq}
\end{eqnarray}
for all $\mathbf{k} \in \mathbb{Z}_+^2$.
\end{proposition}

\begin{proof}
Let $\mathbf{k} \in \mathbb{Z}_+^2$; by Lemma \ref{CartAlu},
\begin{eqnarray}
\widetilde{T}_{2}\widetilde{T}_{1}e_{\mathbf{k}}&=&\sqrt{\alpha_{\mathbf{k}}\alpha_{\mathbf{k}+\varepsilon_{1}}\beta_{\mathbf{k}+\varepsilon_{1}}\beta_{\mathbf{k}+\varepsilon_{1}+\varepsilon_{2}}}e_{\mathbf{k}+\varepsilon_{1}+\varepsilon_{2}} \nonumber \\
&=&\sqrt{(\alpha_{\mathbf{k}}\beta_{\mathbf{k}+\varepsilon_{1}})\alpha_{\mathbf{k}+\varepsilon_{1}}\beta_{\mathbf{k}+\varepsilon_{1}+\varepsilon_{2}}}e_{\mathbf{k}+\varepsilon_{1}+\varepsilon_{2}} \nonumber \\
&=&\sqrt{(\beta_{\mathbf{k}}\alpha_{\mathbf{k}+\varepsilon_{2}})\alpha_{\mathbf{k}+\varepsilon_{1}}\beta_{\mathbf{k}+\varepsilon_{1}+\varepsilon_{2}}}e_{\mathbf{k}+\varepsilon_{1}+\varepsilon_{2}} \; \; (\textrm{by } \ref{commuting}) \nonumber \\
&=&\sqrt{\beta_{\mathbf{k}}\alpha_{\mathbf{k}+\varepsilon_{1}}(\beta_{\mathbf{k}+\varepsilon_{2}}\alpha_{\mathbf{k}+2\varepsilon_{2}})}e_{\mathbf{k}+\varepsilon_{1}+\varepsilon_{2}} \; \; (\textrm{again by } \ref{commuting}) \nonumber \\
&=&\sqrt{\beta_{\mathbf{k}}\beta_{\mathbf{k}+\varepsilon_{2}}(\alpha_{\mathbf{k}+\varepsilon_{1}}\alpha_{\mathbf{k}+2\varepsilon_{2}}})e_{\mathbf{k}+\varepsilon_{1}+\varepsilon_{2}}. \label{eqq1}
\end{eqnarray}
On the other hand,
\begin{eqnarray}
\widetilde{T}_{1}\widetilde{T}_{2}e_{\mathbf{k}}&=&\sqrt{\beta_{\mathbf{k}}\beta_{\mathbf{k}+\varepsilon_{2}}\alpha_{\mathbf{k}+\varepsilon_{2}}\alpha_{\mathbf{k}+\varepsilon_{1}+\varepsilon_{2}}}e_{\mathbf{k}+\varepsilon_{1}+\varepsilon_{2}}.  \label{eqq2}
\end{eqnarray}
From (\ref{eqq1}) and (\ref{eqq2}) it follows that $\widetilde{T}_{1}\widetilde{T}_{2}=\widetilde{T}_{2}\widetilde{T}_{1}$ if and only if 
$$
\alpha _{\mathbf{k+}\varepsilon _{2}}\alpha _{\mathbf{k}+\varepsilon _{1}+\varepsilon _{2}}=\alpha_{\mathbf{k+}\varepsilon _{1}}\alpha _{\mathbf{k}+2\varepsilon _{2}},
$$
as desired.
\end{proof}

\begin{remark}
\label{Re 1} By Proposition \ref{commuting1} and the commutativity condition for $W_{(\alpha ,\beta )}$, it is straightforward to prove that (\ref{prop1eq}) is equivalent to
\begin{equation}
\beta _{\mathbf{k+}\varepsilon _{1}}\beta _{\mathbf{k}+\varepsilon _{1}+\varepsilon _{2}}=\beta_{\mathbf{k+}\varepsilon _{2}}\beta _{\mathbf{k}+2\varepsilon _{1}}.
\label{equ3}
\end{equation}
for all $\mathbf{k} \in \mathbb{Z}_+^2$. \; \qed
\end{remark}

\begin{lemma}
(\cite{CLY1})\label{khypo} Let $W_{(\alpha ,\beta )}$ be a commuting $2$-variable weighted shift. \ Then the following are equivalent:\newline
(i) $\ W_{(\alpha ,\beta )}$ is $k$-hyponormal;\newline
(ii) $\ M_{\mathbf{u}}(k)\left( W_{(\alpha ,\beta )}\right) :=\left( \gamma
_{\mathbf{u}+(n,m)+(p,q)}\left( W_{(\alpha ,\beta )}\right) \right)
_{_{0\leq p+q\leq k}^{0\leq n+m\leq k}}\geq 0$ for all $\mathbf{u}\in
\mathbb{Z}_{+}^{2}$.
\end{lemma}

We recall that $\mathcal{M}_{i}$ $\left( \text{resp. }\mathcal{N}_{j}\right)
$ is the subspace of $\ell ^{2}(\mathbb{Z}_{+}^{2})$ spanned by the
canonical orthonormal basis associated to indices $\mathbf{k}=(k_{1},k_{2})$
with $k_{1}\geq 0$ and $k_{2}\geq i$ (resp. $k_{1}\geq j$ 
\newline and $k_{2}\geq 0$). \ For simplicity, we write $\mathcal{M}=\mathcal{M}_1$ and $\mathcal{N}=\mathcal{N}_1$. \ The \textit{core} $c(W_{(\alpha ,\beta )})$ of $%
W_{(\alpha ,\beta )}$ is the restriction of $W_{(\alpha ,\beta )}$ to the
invariant subspace $\mathcal{M}\cap \mathcal{N}$. \ A $2$-variable weighted
shift $W_{(\alpha ,\beta )}$ is said to be of\textit{\ tensor form} if it is
of the form $(I\otimes W_{\sigma },W_{\tau }\otimes I)$ for suitable $1$-variable weight sequences $\sigma$ and $\tau$. \ We also let 
$$
\mathcal{TC}:=\{W_{(\alpha ,\beta )}: c(W_{(\alpha ,\beta )}) \textrm{ is of tensor form} \}.
$$

\begin{proposition}
\label{propscaling} (cf. \cite{KimYoon}) \ Let $W_{(\alpha ,\beta
)}\equiv \left( T_{1},T_{2}\right)$ be a commuting $2$-variable weighted shift. \ Then, for $a,b>0$ and $k \ge 1$
we have%
\begin{equation*}
\left( T_{1},T_{2}\right) \text{ is } k \text{-hyponormal} \Longleftrightarrow \left(
aT_{1},bT_{2}\right) \text{ is } k \text{-hyponormal}.
\end{equation*}
\end{proposition}

\subsection{Hyponormality is not preserved under the toral Aluthge transform} \label{subsec31} \ As we observed in the Introduction, the $1$-variable Aluthge transform leaves the class of hyponormal weighted shifts invariant. \ In this Subsection we will show that the same is not true of the toral Aluthge transform acting on $2$-variable weighted shifts. \ To see this, consider the commuting $2$-variable weighted shift $W_{(\alpha ,\beta )}$ given by Figure \ref{Figure 2ex}(ii). \ That is, $W_{(\alpha ,\beta )}$ has a symmetric weight diagram, has a core of tensor form (with Berger measure $\xi \times \xi$), and with zero-th and first rows given by backward extensions of a weighted shift whose Berger measure is $\xi$; we will denote those backward extensions by $[x_0,\xi]$ and $[a,\xi]$, respectively. \ Also, we denote by $\omega_0,\omega_1,\cdots$ the weight sequence associated with $\xi$. \ Since we wish to characterize the subnormality of $W_{(\alpha ,\beta )}$, we assume that $[x_0,\xi]$ and $[a,\xi]$ are subnormal, which requires that $\frac{1}{s} \in L^1(\xi)$. \ Let $\rho:=\int{\frac{1}{s} d\xi(s)}<\infty$. \ We recall the following result from \cite{CuYo1}.

\setlength{\unitlength}{1mm} \psset{unit=1mm}
\begin{figure}[th]
\begin{center}
\begin{picture}(120,40)

\pspolygon*[linecolor=lightgray](25,16)(57,16)(57,38)(25,38)

\psline{->}(10,6)(58,6)
\psline(10,16)(57,16)
\psline(10,26)(57,26)
\psline(10,36)(57,36)
\psline{->}(10,6)(10,38)
\psline(25,6)(25,37)
\psline(40,6)(40,37)
\psline(55,6)(55,37)

\put(2.3,3.2){\footnotesize{$(0,0)$}}
\put(9,-2){$(i)$}
\put(21,3){\footnotesize{$(1,0)$}}
\put(36,3){\footnotesize{$(2,0)$}}
\put(51,3){\footnotesize{$(3,0)$}}

\put(14,7){\footnotesize{$x_{0}$}}
\put(29,7){\footnotesize{$x_{1}$}}
\put(44,7){\footnotesize{$x_2=\omega_{1}$}}
\put(56,7){\footnotesize{$\omega_{2}$}}

\put(14,17){\footnotesize{$a$}}
\put(29,17){\footnotesize{$\omega_{0}$}}
\put(44,17){\footnotesize{$\omega_{1}$}}
\put(56,17){\footnotesize{$\omega_{2}$}}

\put(14,28){\footnotesize{$a \frac{\omega_{0}}{x_{1}}$}}
\put(29,27){\footnotesize{$\omega_{0}$}}
\put(44,27){\footnotesize{$\omega_{1}$}}
\put(56,27){\footnotesize{$\omega_{2}$}}

\put(15,38){\footnotesize{$a \frac{\omega_{0}}{x_{1}}$}}
\put(30,37){\footnotesize{$\omega_{0}$}}
\put(45,37){\footnotesize{$\omega_{1}$}}
\put(56,37){\footnotesize{$\omega_{2}$}}

\psline{->}(25,1)(40,1)
\put(30,-2.5){$\rm{T}_1$}
\psline{->}(2, 15)(2,30)
\put(-3,20){$\rm{T}_2$}

\put(2.7,15){\footnotesize{$(0,1)$}}
\put(2.7,25){\footnotesize{$(0,2)$}}
\put(2.7,35){\footnotesize{$(0,3)$}}

\put(10,11){\footnotesize{$y_{0}$}}
\put(10,20){\footnotesize{$y_{1}$}}
\put(10,31){\footnotesize{$y_2=\tau_{1}$}}
\put(11,37){\footnotesize{$\vdots$}}

\put(25,11){\footnotesize{$a \frac{y_{0}}{x_{0}}$}}
\put(25,20){\footnotesize{$\tau_{0}$}}
\put(25,31){\footnotesize{$\tau_{1}$}}
\put(26,37){\footnotesize{$\vdots$}}

\put(40,11){\footnotesize{$a \frac{y_0 \omega_{0}}{x_0 x_{1}}$}}
\put(40,20){\footnotesize{$\tau_{0}$}}
\put(40,31){\footnotesize{$\tau_{1}$}}
\put(41,37){\footnotesize{$\vdots$}}

\pspolygon*[linecolor=lightgray](67,16)(114,16)(114,38)(67,38)
\pspolygon*[linecolor=lightgray](82,6)(114,6)(114,38)(82,38)

\psline{->}(67,6)(115,6)
\psline(67,16)(114,16)
\psline(67,26)(115,26)
\psline(67,36)(114,36)
\psline{->}(67,6)(67,38)
\psline(82,6)(82,37)
\psline(97,6)(97,37)
\psline(112,6)(112,37)

\put(65,3.2){\footnotesize{$(0,0)$}}
\put(67.2,-2){$(ii)$}
\put(78,3){\footnotesize{$(1,0)$}}
\put(93,3){\footnotesize{$(2,0)$}}
\put(108,3){\footnotesize{$(3,0)$}}

\put(73,7){\footnotesize{$x_{0}$}}
\put(87,7){\footnotesize{$\omega_{0}$}}
\put(101,7){\footnotesize{$\omega_{1}$}}
\put(113,7){\footnotesize{$\cdots$}}

\put(73,17){\footnotesize{$a$}}
\put(87,17){\footnotesize{$\omega_{0}$}}
\put(101,17){\footnotesize{$\omega_{1}$}}
\put(113,17){\footnotesize{$\cdots$}}

\put(73,28){\footnotesize{$a$}}
\put(87,27){\footnotesize{$\omega_{0}$}}
\put(101,27){\footnotesize{$\omega_{1}$}}
\put(113,27){\footnotesize{$\cdots$}}

\put(73,36.5){\footnotesize{$\cdots$}}
\put(87,36.5){\footnotesize{$\cdots$}}
\put(102,36.5){\footnotesize{$\cdots$}}

\psline{->}(83,1)(98,1)
\put(88,-2.5){$\rm{T}_1$}

\put(67,10){\footnotesize{$x_{0}$}}
\put(67,20){\footnotesize{$\omega_{0}$}}
\put(67,30){\footnotesize{$\omega_{1}$}}
\put(68,37){\footnotesize{$\vdots$}}

\put(82,11){\footnotesize{$a$}}
\put(82,20){\footnotesize{$\omega_{0}$}}
\put(82,30){\footnotesize{$\omega_{1}$}}
\put(83,37){\footnotesize{$\vdots$}}

\put(97,11){\footnotesize{$a$}}
\put(97,20){\footnotesize{$\omega_{0}$}}
\put(97,30){\footnotesize{$\omega_{1}$}}
\put(98,37){\footnotesize{$\vdots$}}


\end{picture}
\end{center}
\caption{Weight diagram of a 2-variable weighted shift $W_{(\protect\alpha,%
\protect\beta)}\equiv(T_{1},T_{2})$ with core of tensor form and with
commutative toral transform. \ Observe that $x_2=\protect\omega_1$, $x_3=%
\protect\omega_2$, $\cdots$, $y_2=\protect\tau_1$, $y_3=\protect\tau_2$, $%
\cdots$, and $\protect\tau_0 x_1 = \protect\omega_0 y_1$ all follow from (%
\protect\ref{prop1eq}). Weight diagram of the 2-variable weighted shift $W_{(%
\protect\alpha,\protect\beta)}\equiv(T_{1},T_{2})$ in Subsection \protect\ref%
{subsec31}.}
\label{Figure 2ex}
\end{figure}
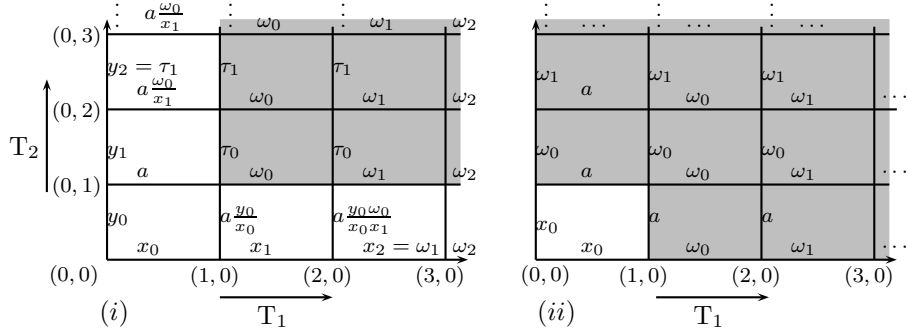

\begin{lemma}
\label{backwardext}  \ (Subnormal backward extension of a $1$-variable weighted
shift \cite[Proposition 1.5]{CuYo1}) \ Let $T$ be a weighted shift whose restriction $T_{%
\mathcal{L}}$ to $\mathcal{L}=\vee \{e_{1},e_{2},\cdots \}$ is subnormal,
with associated measure $\mu _{\mathcal{L}}.$ \ Then $T$ is subnormal (with
associated measure $\mu $) if and only if\newline
(i) $\ \frac{1}{t}\in L^{1}(\mu _{\mathcal{L}})$\newline
(ii) $\ \alpha _{0}^{2}\leq (\left\| \frac{1}{t}\right\| _{L^{1}(\mu _{%
\mathcal{L}})})^{-1}$\newline
In this case, $d\mu (t)=\frac{\alpha _{0}^{2}}{t}d\mu _{\mathcal{L}%
}(t)+(1-\alpha _{0}^{2}\left\| \frac{1}{t}\right\| _{L^{1}(\mu _{\mathcal{L}%
})})d\delta _{0}(t)$. \ In
particular, $T$ is never subnormal when $\mu _{\mathcal{L}}(\{0\})>0$. \
\end{lemma}

Thus, by Lemma \ref{backwardext}, we must have $x_0^2 \rho \le 1$ and $a^2 \rho \le 1$.
\ For the proof of Lemma \ref{lem39}, we need to recall a few facts about $2$-variable weighted shifts.

\begin{lemma} \ (cf. \cite{CuYo2}) \newline
(i) \ Let $\mu $ and $\nu $ be two positive measures on a set $X$. \ We say
that $\mu \leq \nu $ on $X,$ if $\mu (E)\leq \nu (E)$ for all Borel subset $%
E\subseteq X$; equivalently, $\mu \leq \nu $ if and only if $\int fd\mu \leq
\int fd\nu $ for all $f\in C(X)$ such that $f\geq 0$ on $X$.\newline
(ii)\ \ Let $\mu $ be a probability measure on $X\times Y$, and assume that $%
\frac{1}{t}\in L^{1}(\mu ).$ \ The \textit{extremal measure} $\mu _{ext}$
(which is also a probability measure) on $X\times Y$ is given by $d\mu
_{ext}(s,t):=(1-\delta _{0}(t))\frac{1}{t\left\Vert \frac{1}{t}\right\Vert
_{L^{1}(\mu )}}d\mu (s,t)$. \newline
(iii) \ Given a measure $\mu $ on $X\times Y$, the \textit{marginal measure}
$\mu ^{X}$ is given by $\mu ^{X}:=\mu \circ \pi _{X}^{-1}$, where $\pi
_{X}:X\times Y\rightarrow X$ is the canonical projection onto $X$. \ Thus, $%
\mu ^{X}(E)=\mu (E\times Y)$, for every $E\subseteq X$.
\end{lemma}

\begin{lemma}
\label{backext} (\cite[Proposition 3.10]{CuYo1}) \ (Subnormal backward
extension of a $2$-variable weighted shift) \ Assume that $W_{(\alpha ,\beta )}$ is a commuting pair of hyponormal operators, and that $W_{(\alpha ,\beta )}|_{\mathcal{M}}$ is subnormal with associated measure $%
\mu _{\mathcal{M}}$. \ Then, $W_{(\alpha ,\beta )}$ is subnormal if and only
if the following conditions hold:\newline
$(i)$ $\ \frac{1}{t}\in L^{1}(\mu _{\mathcal{M}})$;\newline
$(ii)$ $\ \beta _{00}^{2}\leq (\left\Vert \frac{1}{t}\right\Vert _{L^{1}(\mu
_{\mathcal{M}})})^{-1}$;\newline
$(iii)$ $\ \beta _{00}^{2}\left\Vert \frac{1}{t}\right\Vert _{L^{1}(\mu _{%
\mathcal{M}})}(\mu _{\mathcal{M}})_{ext}^{X}\leq \xi _{0}$.\newline
Moreover, if $\beta _{00}^{2}\left\Vert \frac{1}{t}\right\Vert _{L^{1}(\mu _{%
\mathcal{M}})}=1,$ then $(\mu _{\mathcal{M}})_{ext}^{X}=\xi _{0}$. \ In the
case when $W_{(\alpha ,\beta )}$ is subnormal, the Berger measure $\mu $ of $%
W_{(\alpha ,\beta )}$ is given by%
\begin{eqnarray}
d\mu (s,t) & = & \beta _{00}^{2}\left\Vert \frac{1}{t}\right\Vert _{L^{1}(\mu _{\mathcal{M}%
})}d(\mu _{\mathcal{M}})_{ext}(s,t) \nonumber \\
&& + (d\xi _{0}(s)-\beta _{00}^{2}\left\Vert
\frac{1}{t}\right\Vert _{L^{1}(\mu _{\mathcal{M}})}d(\mu _{\mathcal{M}%
})_{ext}^{X}(s))d\delta _{0}(t).
\label{Berger}
\end{eqnarray}
\end{lemma}

In the rest of this section, we restrict attention to the $2$-variable weighted shift with weight diagram given as in Figure \ref{Figure 2ex}(ii). 

\begin{lemma} \label{lem39}
Let $W_{(\alpha ,\beta )}$ be a $2$-variable weighted shift, let $\rho:=\int{\frac{1}{s} d\xi(s)}<\infty$, and assume that $x_0^2 \rho \le 1$ and $a^2 \rho \le 1$. \ Then, $W_{(\alpha ,\beta )}$ is subnormal if and only if $x_0^2 \rho(2-a^2 \rho)\le 1$.
\end{lemma}

\begin{proof}
Observe that the Berger measure of $[x_0,\xi]$ is $\xi_{x_0} =\frac{x_0^2 \xi}{s} +(1-x_0^2 \rho) \delta_0$ and similarly the Berger measure of $[a,\xi]$ is $\xi_{a} =\frac{a^2 \xi}{s} +(1-a^2 \rho) \delta_0$. \ The restriction of $W_{(\alpha ,\beta )}$ to the subspace $\mathcal{M}$ is then $\mu_{\mathcal{M}}=\xi_a \times \xi$, from which it follows at once that $(\mu_{\mathcal{M}})_{ext}^X=\xi_a$. \ Therefore, for the subnormality of $W_{(\alpha ,\beta )}$ we will need $x_0^2 \rho \xi_a \le \xi_{x_0}$ and this naturally leads to the condition $x_0^2 \rho(2-a^2 \rho)\le 1$, as desired.
\end{proof}

\begin{lemma}
The 2-variable weighted shift $\widetilde{W}_{(\alpha ,\beta )}$ is hyponormal if and only if $|a-x_0|\le \omega_1-x_0$.
\end{lemma}

\begin{proof}
Since the restrictions of $\widetilde{W}_{(\alpha ,\beta )}$ to the subspaces $\mathcal{M}$ and $\mathcal{N}$ are subnormal, Lemma \ref{khypo} says that $\widetilde{W}_{(\alpha ,\beta )}$ is hyponormal if and only if $M_{(0,0)}(1) \equiv M_{(0,0)}(1)(\widetilde{W}_{(\alpha ,\beta )}) \ge 0$. \ Since
$$
M_{(0,0)}(1)=\left(
\begin{array}{cc}
\omega_0 \omega_1 - x_0 \omega_0 & a \omega_0 - x_0 \omega_0 \\
a \omega_0 - x_0 \omega_0 & \omega_0 \omega_1 - x_0 \omega_0
\end{array}
\right),
$$
it follows that $\widetilde{W}_{(\alpha ,\beta )}$ is hyponormal if and only if $|a-x_0| \le \omega_1-x_0$, as desired.
\end{proof}

We observe that if $a \ge x_0$, then $|a-x_0|\le \omega_1-x_0$ becomes $a \le \omega_1$, which is always true. \ Thus, to build an example where the hyponormality of $\widetilde{W}_{ (\alpha ,\beta )}$ is violated, we must necessarily assume that $a<x_0$. \ Incidentally, this assumption automatically leads to $a^2 \rho \le x_0^2 \rho$, so that the subnormality of $W_{(\alpha ,\beta )}$ is now determined by the conditions $x_0^2 \rho \le 1$ and $x_0^2 \rho(2-a^2 \rho)\le 1$. \ In short, an example with the desired properties can be constructed once we guarantee the following three conditions:
\begin{equation}
x_0^2 \rho \le 1 \label{condition1}
\end{equation}
\begin{equation}
x_0^2 \rho (2-a^2 \rho) \le 1 \label{condition2}
\end{equation}
\begin{equation}
x_0 > \frac{\omega_1+a}{2} \label{condition3}.
\end{equation}

Notice that $2-a^2 \rho <2$, so if we were to assume that $x_0^2 \rho \le \frac{1}{2}$ then both conditions (\ref{condition1}) and (\ref{condition2}) would be simultaneously satisfied. \ Moreover, if we were to assume that $x_0 > \frac{\omega_1}{2}$, then we could always find $a<x_0$ such that $x_0>\frac{\omega_1+a}{2}$. \ We can then focus on the following question:

Can we simultaneously guarantee $x_0^2 \rho \le \frac{1}{2}$ and $x_0^2 > \frac{\omega_1^2}{4}$?

Alternatively, we need
\begin{equation}
\frac{\omega_1^2}{4} < x_0^2 \le \frac{1}{2 \rho}. \label{condition4}
\end{equation}

Now, if $\frac{\omega_1^2}{4} < \frac{1}{2 \rho}$, then it would be possible to select $x_0$ such that (\ref{condition4}) is satisfied. \ We have thus established the following result.

\begin{theorem} \label{example100}
Let $W_{(\alpha ,\beta )}$ be as above, and assume that
$$
\omega_1^2 \rho < 2.
$$
Then: (i) \ $W_{(\alpha ,\beta )}$ is subnormal; and (ii) \ $\widetilde{W}_{(\alpha ,\beta )}$ is not hyponormal.
\end{theorem}

We will now show that the condition in Theorem \ref{example100} holds for a large class of $2$-variable weighted shifts. \

\begin{example}
Consider the case when the measure $\xi$ is $2$-atomic, that is, $\xi \equiv r \delta_p + s \delta_q$, with $r,s>0$, $r+s=1$ and $0<p<q$. \ (Recall that $0$ cannot be in the support of $\xi$, because otherwise $\frac{1}{s} \notin L^1(\xi)$.) \ We compute
$$
\omega_1^2 \rho = \frac{rp^2+sq^2}{rp+sq}(\frac{r}{p}+\frac{s}{q})=\frac{r+s (\frac{q}{p})^2}{r+s\frac{q}{p}}(r+\frac{s}{\frac{q}{p}}).
$$
Thus, without loss of generality we can always assume that $p=1$, that is,
$$
\omega_1^2 \rho=\frac{r+sq^2}{r+sq}(r+\frac{s}{q}).
$$
A calculation using {\it Mathematica} \cite{Wol} reveals that for $1 < q \le \tilde{q}:=1/2 + \sqrt{2} + \frac{1}{2} (\sqrt{5 + 4 \sqrt{2}}) \approx 3.546$, we have $\frac{r+sq^2}{r+sq}(r+\frac{s}{q})<2$ for all $r,s>0$ with $r+s=1$. \ As a matter of fact, there is a region $R$ in the $(s,q)$-plane bounded by the graph $q=f(s)$ of a positive convex function $f$, such that $\omega_1^2 \rho < 2$ precisely when $1<q<f(s)$; $R$ contains the rectangle $[0,1] \times (1,\tilde{q}]$. \qed
\end{example}

We have thus established the existence of subnormal $2$-variable weighted shifts $W_{(\alpha ,\beta )}$ with non-hyponormal toral Aluthge transforms.



\section{\label{Sec2}The spherical Aluthge Transform}

In this section, we study the second plausible definition of the multivariable Aluthge transform, which we will denote, to avoid confusion, by
$\widehat{(T_{1},T_{2})}$; this corresponds to (\ref{Def-Alu1}). \ We begin with the following elementary result.

\begin{proposition} \label{basic}
Assume that $(T_1,T_2) \equiv (V_1P,V_2P)$, where $P=(T_1^*T_1+T_2^*T_2)^{1/2}$, and let $\widehat {(T_1,T_2)}\equiv (\widehat{T_1},\widehat{T_2}):=(\sqrt{P}V_1\sqrt{P},\sqrt{P}V_2\sqrt{P})$. \ Assume also that $(T_1,T_2)$ is commutative. \ Then \newline
(i) \ $(V_1,V_2)$ is a (joint) partial isometry; more precisely, $V_1^*V_1+V_2^*V_2$ is the projection onto $\emph{ran} \; P$; \newline
(ii) \ $\widehat {(T_1,T_2)}$ is commutative on $\emph{ran} \;P$, so in particular $\widehat {(T_1,T_2)}$ is commutative whenever $P$ is injective.
\end{proposition}

\begin{proof} \ (i) An easy computation reveals that
$$
P^2=T_1^*T_1+T_2^*T_2=(V_1P)^*(V_1P)+(V_2P)^*(V_2P)=P(V_1^*V_1+V_2^*V_2)P,
$$
and therefore $(V_1^*V_1+V_2^*V_2)|_{\textrm{ran} \; P}$ is the identity operator on $\textrm{ran} \; P$, as desired.

To prove (ii), consider the product
$$
\widehat{T_1}\widehat{T_2}=\sqrt{P}V_1\sqrt{P}\sqrt{P}V_2\sqrt{P}=\sqrt{P}V_1PV_2\sqrt{P}.
$$
Then
\begin{eqnarray*}
\widehat{T_{1}}\widehat{T_{2}}\sqrt{P} &=&\sqrt{P}T_{1}T_{2}=\sqrt{P}%
T_{2}T_{1}=(\sqrt{P}V_{2}PV_{1}\sqrt{P})\sqrt{P} \\
&=&(\sqrt{P}V_{2}\sqrt{P})(\sqrt{P}V_{1}\sqrt{P})\sqrt{P}=\widehat{T_{2}}%
\widehat{T_{1}}\sqrt{P}.
\end{eqnarray*}
It follows at once that $\widehat{T_1}\widehat{T_2}-\widehat{T_2}\widehat{T_1}$ vanishes on $\textrm{ran} \; P$, as desired.
\end{proof}

We now prove:

\begin{proposition}
\label{Prop1}Given a $2$-variable weighted shift $W_{(\alpha ,\beta )}\equiv
(T_{1},T_{2})$, let $\widehat{W}_{(\alpha ,\beta )}$ be given by (\ref{Def-Alu1}). \ Assume that $W_{(\alpha ,\beta )}$ is commutative. \ Then $\widehat{W}_{(\alpha ,\beta )}$ is commutative.
\end{proposition}

\begin{proof}
Straightforward from Proposition \ref{basic}.
\end{proof}

We briefly pause to describe how $\widehat{W}_{(\alpha ,\beta )}$ acts on the canonical orthonormal basis vectors.

\begin{lemma} \label{PolarAlu}
Let $W_{(\alpha ,\beta )} \equiv \left(T_{1},T_{2}\right)$ be a $2$-variable weighted shift. \ Then
$$
\widehat{T}_{1}e_{\mathbf{k}}=\alpha_{\mathbf{k}} \frac{(\alpha_{\mathbf{k}+\mathbf{\epsilon}_1}^2+\beta_{\mathbf{k}+\mathbf{\epsilon}_1}^2)^{1/4}}{(\alpha_{\mathbf{k}}^2+\beta_{\mathbf{k}}^2)^{1/4}} e_{\mathbf{k}+\mathbf{\epsilon}_1}
$$
and
$$
\widehat{T}_{2}e_{\mathbf{k}}=\beta_{\mathbf{k}} \frac{(\alpha_{\mathbf{k}+\mathbf{\epsilon}_2}^2+\beta_{\mathbf{k}+\mathbf{\epsilon}_2}^2)^{1/4}}{(\alpha_{\mathbf{k}}^2+\beta_{\mathbf{k}}^2)^{1/4}} e_{\mathbf{k}+\mathbf{\epsilon}_2}
$$
for all $\mathbf{k} \in \mathbb{Z}_+^2$.
\end{lemma}

\begin{proof}
Straightforward from (\ref{Def-Alu1}).
\end{proof}

\setlength{\unitlength}{1mm} \psset{unit=1mm}
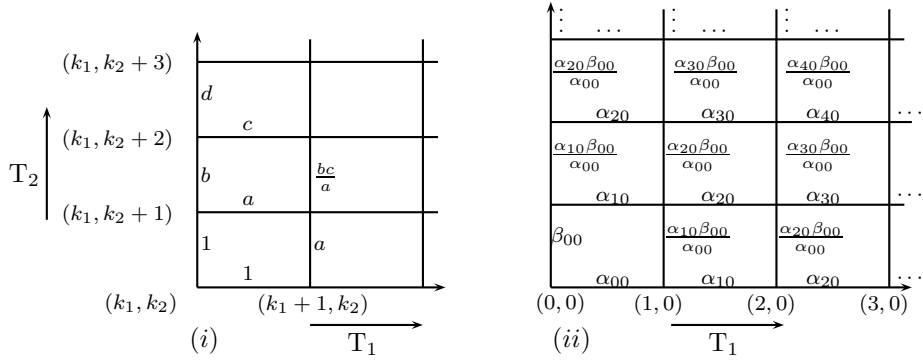
\begin{figure}[th]
\begin{center}
\begin{picture}(120,45)

\psline{->}(20,6)(53,6)
\psline(20,16)(52,16)
\psline(20,26)(52,26)
\psline(20,36)(52,36)
\psline{->}(20,6)(20,40)
\psline(35,6)(35,39)
\psline(50,6)(50,39)

\put(7.7,3.2){\footnotesize{$(k_1,k_2)$}}
\put(19,-2){$(i)$}
\put(28,3){\footnotesize{$(k_1+1,k_2)$}}

\put(26,7){\footnotesize{$1$}}

\put(26,17){\footnotesize{$a$}}

\put(26,27){\footnotesize{$c$}}

\psline{->}(35,1)(50,1)
\put(40,-2.5){$\rm{T}_1$}
\psline{->}(0, 15)(0,30)
\put(-5,20){$\rm{T}_2$}

\put(2,15){\footnotesize{$(k_1,k_2+1)$}}
\put(2,25){\footnotesize{$(k_1,k_2+2)$}}
\put(2,35){\footnotesize{$(k_1,k_2+3)$}}

\put(20.5,11){\footnotesize{$1$}}
\put(20.5,20){\footnotesize{$b$}}
\put(20.5,31){\footnotesize{$d$}}

\put(35.5,11){\footnotesize{$a$}}
\put(35.5,20){\footnotesize{$\frac{bc}{a}$}}


\psline{->}(67,6)(115,6)
\psline(67,17)(114,17)
\psline(67,28)(115,28)
\psline(67,39)(114,39)
\psline{->}(67,6)(67,44)
\psline(82,6)(82,43)
\psline(97,6)(97,43)
\psline(112,6)(112,43)

\put(65,3.2){\footnotesize{$(0,0)$}}
\put(67.2,-2){$(ii)$}
\put(78,3){\footnotesize{$(1,0)$}}
\put(93,3){\footnotesize{$(2,0)$}}
\put(108,3){\footnotesize{$(3,0)$}}

\put(73,6.5){\footnotesize{${\alpha_{00}}$}}
\put(87,6.5){\footnotesize{${\alpha_{10}}$}}
\put(101,6.5){\footnotesize{${\alpha_{20}}$}}
\put(113,6.6){\footnotesize{$\cdots$}}

\put(73,17.5){\footnotesize{${\alpha_{10}}$}}
\put(87,17.5){\footnotesize{${\alpha_{20}}$}}
\put(101,17.5){\footnotesize{${\alpha_{30}}$}}
\put(113,17.6){\footnotesize{$\cdots$}}

\put(73,28.5){\footnotesize{${\alpha_{20}}$}}
\put(87,28.5){\footnotesize{${\alpha_{30}}$}}
\put(101,28.5){\footnotesize{${\alpha_{40}}$}}
\put(113,28.6){\footnotesize{$\cdots$}}

\put(73,39.5){\footnotesize{$\cdots$}}
\put(87,39.5){\footnotesize{$\cdots$}}
\put(102,39.5){\footnotesize{$\cdots$}}

\psline{->}(83,1)(98,1)
\put(88,-2.5){$\rm{T}_1$}

\put(67,12){\footnotesize{$\beta_{00}$}}
\put(67,23){\footnotesize{$\frac{\alpha_{10}\beta_{00}}{\alpha_{00}}$}}
\put(67,34){\footnotesize{$\frac{\alpha_{20}\beta_{00}}{\alpha_{00}}$}}
\put(68,40){\footnotesize{$\vdots$}}

\put(82,12){\footnotesize{$\frac{\alpha_{10}\beta_{00}}{\alpha_{00}}$}}
\put(82,23){\footnotesize{$\frac{\alpha_{20}\beta_{00}}{\alpha_{00}}$}}
\put(83,34){\footnotesize{$\frac{\alpha_{30}\beta_{00}}{\alpha_{00}}$}}
\put(83,40){\footnotesize{$\vdots$}}

\put(97,12){\footnotesize{$\frac{\alpha_{20}\beta_{00}}{\alpha_{00}}$}}
\put(98,23){\footnotesize{$\frac{\alpha_{30}\beta_{00}}{\alpha_{00}}$}}
\put(98,34){\footnotesize{$\frac{\alpha_{40}\beta_{00}}{\alpha_{00}}$}}
\put(98,40){\footnotesize{$\vdots$}}

\end{picture}
\end{center}
\caption{Weight diagram of the 2-variable weighted shift in Proposition
\protect\ref{Proposition1} and weight diagram of a commuting $2$-variable weighted shift for which
the toral and spherical Aluthge transforms coincide, respectively.}
\label{Fig 1}
\end{figure}

We next have:

\begin{proposition}
\label{Proposition1} Consider a $2$-variable weighted shift $W_{(\alpha
,\beta )}\equiv \left( T_{1},T_{2}\right)$, and assume that $W_{(\alpha
,\beta )}$ is a commuting pair of hyponormal operators. \ Then so is $\widehat{W}_{(\alpha ,\beta )}$.
\end{proposition}

\begin{proof}
We will establish that $\widehat{T}_2$ is hyponormal. \ Fix a lattice point $(k_1,k_2)$; we would like to prove that $\widehat{\beta}_{(k_1,k_2)} 
\le \widehat{\beta}_{(k_1,k_2+1)}$. \ Since the hyponormality of a Hilbert space operator is invariant under multiplication by a nonzero scalar, we can, without loss of generality, assume that $\alpha_{(k_1,k_2)}=\beta_{(k_1,k_2)}=1$. \ To simplify the calculation, let $a:=\alpha_{(k_1,k_2+1)}$, 
$b:=\beta_{(k_1,k_2+1)}$, $c:=\alpha_{(k_1,k_2+2)}$ and $d:=\beta_{(k_1,k_2+2)}$. \ Thus, the weight diagram of $(T_1,T_2)$ is now given as in Figure %
\ref{Fig 1}(i). \ Since $T_2$ is hyponormal, we must necessarily have 
\begin{equation} \label{bc}
a \le \frac{bc}{a}
\end{equation}
in the first column of the weight diagram in Figure \ref{Fig 1}(i). \ Recall also the Cauchy-Schwarz inequality $a^2b^2 \le \frac{a^4+b^4}{2}$. \ Then 
\begin{eqnarray*}
\widehat{\beta }_{(k_{1},k_{2})}^{4}&=&\frac{a^2+b^2}{2}=\frac{(a^2+b^2)^2}{2(a^2+b^2)}=\frac{a^4+2a^2b^2+b^4}{2(a^2+b^2)} \\
& \le &\frac{a^4+b^4}{a^2+b^2} \le \frac{b^2c^2+b^2d^2}{a^2+b^2} \; \; (\textrm{by } (\ref{bc}) \textrm{ and the fact that } b \le d) \\
& = &b^2 \cdot \frac{c^2+d^2}{a^2+b^2} \le b^4 \cdot \frac{c^2+d^2}{a^2+b^2} = \widehat{\beta }_{(k_{1},k_{2}+1)}^{4},
\end{eqnarray*}
as desired.
\end{proof}

We now present an example of a hyponormal $2$-variable weighted shift $W_{(\alpha
,\beta )}$ for which $\widetilde{W}_{(\alpha
,\beta )}$ is not hyponormal. \ While we have already encountered this behavior (cf. Theorem \ref{example100}), the simplicity of the following example warrants special mention (aware as we are that the result is weaker than Theorem \ref{example100}). \ Moreover, this example shows that the spherical Aluthge transform $\widehat{W}_{(\alpha,\beta )}$ may be hyponormal even if $W_{(\alpha,\beta )}$ is not. 

\begin{example}
\label{2 atomic}  For $0<x,y<1$, let $W_{(\alpha ,\beta )}$ be the $2$-variable weighted shift in Figure \ref{Figure 2ex}(ii), where $\omega_0=\omega_1=\omega_2=\cdots:=1$, $x_{0}:=x$, $a:=y$. \ Then \newline
(i) \ $\ W_{(\alpha ,\beta )}$ is subnormal $\Longleftrightarrow x \le s(y):=\sqrt{\frac{1}{2-y^2}}$; \newline
(ii) \ $\ W_{(\alpha ,\beta )}$ is hyponormal $\Longleftrightarrow x\le h(y):=\sqrt{\frac{%
1+y^{2}}{2}}$; \newline
(iii) \ $\ \widetilde{W}_{(\alpha ,\beta )}$ is hyponormal $\Longleftrightarrow x \le CA(y):=\frac{1+y}{2}$; \newline
(iv) \ $\widehat{W}_{(\alpha ,\beta )}$ is hyponormal $\Longleftrightarrow x\le
PA(y):=\frac{2\left( 1+y^{2}-y^{4}\right) }{\left( 1+\sqrt{2}\right) \left(
1+y^{2}\right) \left( \sqrt{1+y^{2}}-y^{2}\right) }$.\newline
Clearly, $s(y) \le h(y) \le PA(y)$ and $CA(y) < h(y)$ for all $0 < y <1$, while $CA(y) < s(y)$ on $(0,q)$ and $CA(y) > s(y)$ on $(q,1)$, where $q \cong 0.52138$. \ Then $W_{(\alpha ,\beta )}$ is hyponormal but $\widetilde{W}%
_{(\alpha ,\beta )}$ is not hyponormal if $0<CA(y)<x\le h(y)$, and $\widehat{W}_{(\alpha ,\beta )}$ is hyponormal but $%
W_{(\alpha ,\beta )}$ is not hyponormal if $0<h(y)<x\le PA(y)$.
\end{example}


\section{\label{Identical Aluthge Transforms}$2$-variable Weighted Shifts with Identical Aluthge Transforms}

We shall now characterize the class $\mathcal{A}_{TS}$ of commuting $2$-variable weighted shifts $W_{(\alpha,\beta)}$ for which the toral and spherical Aluthge transforms agree, that is, $\widetilde{W}_{(\alpha,\beta)}=\widehat{W}_{(\alpha,\beta)}$. \  Using Lemmas \ref{CartAlu} and \ref{PolarAlu}, it suffices to restrict attention to the equalities

$$
\sqrt{\alpha_{\mathbf{k}}\alpha_{\mathbf{k}+\mathbf{\varepsilon}_{1}}}=\alpha_{\mathbf{k}} \frac{(\alpha_{\mathbf{k}+\mathbf{\epsilon}_1}^2+\beta_{\mathbf{k}+\mathbf{\epsilon}_1}^2)^{1/4}}{(\alpha_{\mathbf{k}}^2+\beta_{\mathbf{k}}^2)^{1/4}}
$$
and
$$
\sqrt{\beta_{\mathbf{k}}\beta_{\mathbf{k}+\mathbf{\varepsilon}_{2}}}=\beta_{\mathbf{k}} \frac{(\alpha_{\mathbf{k}+\mathbf{\epsilon}_2}^2+\beta_{\mathbf{k}+\mathbf{\epsilon}_2}^2)^{1/4}}{(\alpha_{\mathbf{k}}^2+\beta_{\mathbf{k}}^2)^{1/4}}
$$
for all $\mathbf{k} \in \mathbb{Z}_+^2$. \ Thus, we easily see that $\widetilde{W}_{(\alpha,\beta)}=\widehat{W}_{(\alpha,\beta)}$ if and only if
$$
\alpha_{\mathbf{k}+\mathbf{\varepsilon}_{1}}^2 (\alpha_{\mathbf{k}}^2+\beta_{\mathbf{k}}^2)=\alpha_{\mathbf{k}}^2(\alpha_{\mathbf{k}+\mathbf{\epsilon}_1}^2+\beta_{\mathbf{k}+\mathbf{\epsilon}_1}^2)
$$
and
$$
\beta_{\mathbf{k}+\mathbf{\varepsilon}_{2}}^2 (\alpha_{\mathbf{k}}^2+\beta_{\mathbf{k}}^2)=\beta_{\mathbf{k}}^2(\alpha_{\mathbf{k}+\mathbf{\epsilon}_2}^2+\beta_{\mathbf{k}+\mathbf{\epsilon}_2}^2)
$$
for all $\mathbf{k} \in \mathbb{Z}_+^2$,
which is equivalent to
$$
\alpha_{\mathbf{k}+\mathbf{\varepsilon}_{1}}\beta_{\mathbf{k}}=\alpha_{\mathbf{k}}\beta_{\mathbf{k}+\mathbf{\varepsilon}_{1}}
$$
and
$$
\beta_{\mathbf{k}+\mathbf{\varepsilon}_{2}}\alpha_{\mathbf{k}}=\beta_{\mathbf{k}}\alpha_{\mathbf{k}+\mathbf{\varepsilon}_{2}}
$$
for all
$\mathbf{k} \in \mathbb{Z}_+^2$. \ If we now recall condition (\ref{commuting}) for the commutativity of $W_{(\alpha,\beta)}$, that is, $\alpha_{\mathbf{k}}\beta_{\mathbf{k}+\mathbf{\epsilon_1}}=\beta_{\mathbf{k}}\alpha_{\mathbf{k}+\mathbf{\epsilon_2}}$ for all $\mathbf{k} \in \mathbb{Z}_+^2$, we see at once that $\widetilde{W}_{(\alpha,\beta)}=\widehat{W}_{(\alpha,\beta)}$ if and only if
$\alpha_{\mathbf{k}+\mathbf{\epsilon_1}}=\alpha_{\mathbf{k}+\mathbf{\epsilon_2}}$ and $\beta_{\mathbf{k}+\mathbf{\epsilon_2}}=\beta_{\mathbf{k}+\mathbf{\epsilon_1}}$ for all
$\mathbf{k} \in \mathbb{Z}_+^2$. \ It follows that the weight diagram for $W_{(\alpha,\beta)}$ is completely determined by the zeroth row and the weight $\beta_{(0,0)}$. \ For, referring to Figure \ref{Figure 1}(i), once we have $\alpha_{(0,0)}$ and $\alpha_{(1,0)}$, we immediately get $\alpha_{(0,1)} (=\alpha_{(1,0)}$). \ With $\alpha_{(0,0)}$ and $\alpha_{(0,1)}$ known, we use commutativity and $\beta_{(0,0)}$ to calculate $\beta_{(1,0)}$. \ Since $\beta_{(0,1)} = \beta_{(1,0)}$ and $\alpha_{(0,2)}=\alpha_{(1,1)}=\alpha_{(2,0)}$, we can then calculate $\beta_{(1,1)}$ and $\beta_{(2,0)}$. \ A similar reasoning yields all remaining $\alpha_{\mathbf{k}}$'s and $\beta_{\mathbf{k}}$'s.

We will now show that, for the purpose of establishing the invariance of $k$-hyponormality under the Aluthge transform for the class $\mathcal{A}_{TS}$, it is enough to assume that $\beta_{(0,0)}=\alpha_{(0,0)}$. \ This is an immediate consequence of the following well known result.

\begin{lemma} \label{lem51}
Let $T$ be a bounded linear operator on Hilbert space, and let $T\equiv VP$ be its polar decomposition. \ Let $a \equiv |a|e^{i \theta}$ be a complex number written in polar form, and define $T_a:=aT$. \ Then, the polar decomposition of $T_a$ is $(e^{i\theta}V)(|a|P)$. \ As a consequence, $\widetilde{T_a}=a\widetilde{T}$.
\end{lemma}

\begin{remark}
By Lemma \ref{lem51}, to study the toral Aluthge transform of $W_{(\alpha,\beta)} \equiv (T_1,T_2)\in \mathcal{A}_{TS}$ we can multiply $T_2$ by the factor
$\frac{\alpha_{(0,0)}}{\beta_{(0,0)}}$. \ This results in a new $2$-variable weighted shift for which $\alpha_{\mathbf{k}}=\beta_{\mathbf{k}}$ for all $\mathbf{k} \in \mathbb{Z}_+^2$. \ This subclass of $\mathcal{A}_{TS}$ is the central subject of the next section. \ Observe that, while the natural generalization of Lemma \ref{lem51} is not true for the spherical Aluthge transform, it is true when restricted to $\mathcal{A}_{TS}$, since both the toral and spherical Aluthge transforms agree on this class. \qed
\end{remark}


\section{\label{Sect3-1}When is Hyponormality Invariant Under the Toral and Spherical Aluthge Transforms?}

In this section we identify a large class of $2$-variable weighted shifts for which the toral ans spherical Aluthge transforms do preserve hyponormality. \ This is in some sense optimal, since we know that $k$-hyponormality ($k \ge 2$) is not preserved by the $1$-variable Aluthge transform \cite{LLY}, as mentioned in the Introduction. \ Since this class is actually a subclass of $\mathcal{A}_{TS}$ (introduced in Section \ref{Identical Aluthge Transforms}), it follows at once that all the results we establish for the toral Aluthge transform are also true for the spherical Aluthge transform.

We start with some definitions. \ Recall that the core $c(W_{(\alpha ,\beta )})$
of $W_{(\alpha ,\beta )}$ is the restriction of $W_{(\alpha ,\beta )}$ to
the invariant subspace $\mathcal{M} \cap \mathcal{N}$. $\ W_{(\alpha
,\beta )}$ is said to be of\textit{\ tensor form} if it is of the form $%
(I \otimes W_{\sigma },W_{\tau } \otimes I)$ for some unilateral weighted
shifts $W_{\sigma }$ and $W_{\tau }$. \ Consider $\Theta \left( W_{\omega
}\right) \equiv W_{(\alpha ,\beta )}$ on $\ell ^{2}(\mathbb{Z}_{+}^{2})$
given by the double-indexed weight sequences $\alpha _{(k_{1},k_{2})}=\beta
_{(k_{1},k_{2})}:=\omega _{k_{1}+k_{2}}$ for $k_{1},k_{2}\geq 0$. \ It is clear that $\Theta \left( W_{\omega }\right) $ is a commuting pair, and we refer to it as a $2$-variable
weighted shift with \textit{diagonal core} \cite{CLY7}. \ This $2$-variable weighted shift can be represented by the weight diagram in Figure \ref{Figure 6}(i)). \ It is straightforward to observe that the class of shifts of the form $\Theta \left( W_{\omega }\right) $ is simply $\mathcal{A}_{\mathcal{TS}}$ with the extra condition $\beta_{(0,0)}=\alpha_{(0,0)}$. \ (For more on these shifts the reader is referred to \cite{CLY7}). \ Now, we show that the $k$-hyponormality of $W_{\omega }$ implies the $k$%
-hyponormality of $\Theta \left( W_{\omega }\right) $. \ For this, we
present a simple criterion to detect the $k$-hyponormality of weighted
shifts.

\begin{lemma}
(\cite{Cu2})\label{k-hyponormal} \ Let $W_{\alpha }e_{i}=\alpha _{i}e_{i+1}$
$(i\geq 0)$ be a hyponormal weighted shift, and let $k\geq 1$. \ The
following statements are equivalent:\newline
(i) $\ W_{\alpha }$ is $k$-hyponormal;\newline
(ii) \ The matrix
\begin{equation*}
(([W_{\alpha }^{\ast j},W_{\alpha }^{i}]e_{u+j},e_{u+i}))_{i,j=1}^{k}
\end{equation*}%
is positive semi-definite for all $u\geq -1$;\newline
(iii) \ The Hankel matrix
\begin{equation*}
H(k;u)\left( W_{\alpha }\right) :=(\gamma _{u+i+j-2})_{i,j=1}^{k+1}
\end{equation*}%
is positive semi-definite for all $u\geq 0$.
\end{lemma}

\setlength{\unitlength}{1mm} \psset{unit=1mm}
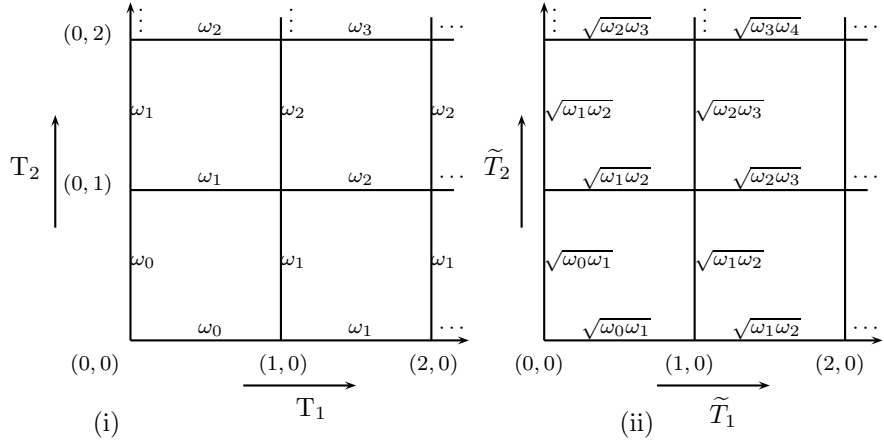
\begin{figure}[th]
\begin{center}
\begin{picture}(135,70)

\psline{->}(20,20)(65,20)
\psline(20,40)(63,40)
\psline(20,60)(63,60)
\psline{->}(20,20)(20,65)
\psline(40,20)(40,63)
\psline(60,20)(60,63)

\put(12,16){\footnotesize{$(0,0)$}}
\put(37,16){\footnotesize{$(1,0)$}}
\put(57,16){\footnotesize{$(2,0)$}}

\put(29,21){\footnotesize{$\omega_{0}$}}
\put(49,21){\footnotesize{$\omega_{1}$}}
\put(61,21){\footnotesize{$\cdots$}}

\put(29,41){\footnotesize{$\omega_{1}$}}
\put(49,41){\footnotesize{$\omega_{2}$}}
\put(61,41){\footnotesize{$\cdots$}}

\put(29,61){\footnotesize{$\omega_{2}$}}
\put(49,61){\footnotesize{$\omega_{3}$}}
\put(61,61){\footnotesize{$\cdots$}}

\psline{->}(35,14)(50,14)
\put(42,10){$\rm{T}_1$}
\psline{->}(10,35)(10,50)
\put(4,42){$\rm{T}_2$}

\put(11,40){\footnotesize{$(0,1)$}}
\put(11,60){\footnotesize{$(0,2)$}}

\put(20,30){\footnotesize{$\omega_{0}$}}
\put(20,50){\footnotesize{$\omega_{1}$}}
\put(21,61){\footnotesize{$\vdots$}}

\put(40,30){\footnotesize{$\omega_{1}$}}
\put(40,50){\footnotesize{$\omega_{2}$}}
\put(41,61){\footnotesize{$\vdots$}}

\put(60,30){\footnotesize{$\omega_{1}$}}
\put(60,50){\footnotesize{$\omega_{2}$}}

\put(15,8){(i)}


\put(85,8){(ii)}

\psline{->}(90,14)(105,14)
\put(97,9){$\widetilde{T}_{1}$}
\psline{->}(72,35)(72,50)
\put(67,42){$\widetilde{T}_{2}$}

\psline{->}(75,20)(120,20)
\psline(75,40)(118,40)
\psline(75,60)(118,60)

\psline{->}(75,20)(75,65)
\psline(95,20)(95,63)
\psline(115,20)(115,63)

\put(71,16){\footnotesize{$(0,0)$}}
\put(91,16){\footnotesize{$(1,0)$}}
\put(111,16){\footnotesize{$(2,0)$}}

\put(80,21){\footnotesize{$\sqrt{\omega_{0}\omega_{1}}$}}
\put(100,21){\footnotesize{$\sqrt{\omega_{1}\omega_{2}}$}}
\put(116,21){\footnotesize{$\cdots$}}

\put(80,41){\footnotesize{$\sqrt{\omega_{1}\omega_{2}}$}}
\put(100,41){\footnotesize{$\sqrt{\omega_{2}\omega_{3}}$}}
\put(116,41){\footnotesize{$\cdots$}}

\put(80,61){\footnotesize{$\sqrt{\omega_{2}\omega_{3}}$}}
\put(100,61){\footnotesize{$\sqrt{\omega_{3}\omega_{4}}$}}
\put(116,61){\footnotesize{$\cdots$}}

\put(75,30){\footnotesize{$\sqrt{\omega_{0}\omega_{1}}$}}
\put(75,50){\footnotesize{$\sqrt{\omega_{1}\omega_{2}}$}}
\put(76,61){\footnotesize{$\vdots$}}

\put(95,30){\footnotesize{$\sqrt{\omega_{1}\omega_{2}}$}}
\put(95,50){\footnotesize{$\sqrt{\omega_{2}\omega_{3}}$}}
\put(96,61){\footnotesize{$\vdots$}}


\end{picture}
\end{center}
\caption{Weight diagram of a generic $2$-variable weighted shift $\Theta
\left( W_{\protect\omega }\right) \equiv \mathbf{(}T_{1},T_{2})$ and weight
diagram of the toral Aluthge transform $\widetilde{\Theta \left( W_{\protect\omega }\right)}\equiv (\widetilde{T}_{1},\widetilde{T}_{2})$ of $\Theta
\left( W_{\protect\omega }\right)$, respectively.}
\label{Figure 6}
\end{figure}

\subsection{Preservation of hyponormality}

We then have:

\begin{proposition}
\label{propscaling2}Consider $\Theta \left( W_{\omega }\right) \equiv
\mathbf{(}T_{1},T_{2})$ given by Figure \ref{Figure 6}(i). \ Then for $%
k\geq 1$
\begin{equation*}
W_{\omega }\text{ is }k\text{-hyponormal if and only if }\Theta \left(
W_{\omega }\right) \text{is } k \text{-hyponormal.}
\end{equation*}
\end{proposition}

\begin{proof}
$(\Longleftarrow )$ \ This is clear from the construction of $\Theta \left(
W_{\omega }\right) $ and Figure \ref{Figure 6}(i).\newline
$(\Longrightarrow )$ \ For $k\geq 1$, we suppose that $W_{\omega }$ is a $k$%
-hyponormal weighted shift. \ Then, by Lemma \ref{k-hyponormal}, for all $%
k_{1}\geq 0$, we have that the Hankel matrix%
\begin{equation*}
H(k;u)\left( W_{\omega }\right) :=(\gamma _{u+i+j-2}\left( W_{\omega
}\right) )_{i,j=1}^{k+1}\geq 0.
\end{equation*}

By Lemma \ref{khypo}, we can see that a $2$-variable weighted shift $%
W_{(\alpha ,\beta )}$ is $k$-hyponormal if and only if
\begin{equation}
M_{\mathbf{u}}(k)\left( W_{(\alpha ,\beta )}\right) =(\gamma _{\mathbf{u}%
+(m,n)+(p,q)})_{_{0\leq p+q\leq k}^{0\leq n+m\leq k}}\geq 0,  \label{k-hy}
\end{equation}%
for all $\mathbf{u}\equiv (u_{1},u_{2})\in \mathbb{Z}_{+}^{2}$. \ Thus, for $%
\Theta \left( W_{\omega }\right)$ $k$-hyponormal, it is enough to show
that $M_{\mathbf{u}}(k)\geq 0$ for all $\mathbf{u}\in \mathbb{Z}_{+}^{2}$. \
Observe that the moments associated with $\Theta \left( W_{\omega }\right) $
are
\begin{equation}
\gamma _{\mathbf{u}}(\Theta \left( W_{\omega }\right) )=\gamma
_{u_{1}+u_{2}}\left( W_{\omega }\right) \equiv \gamma _{u_{1}+u_{2}} \; (\text{all }%
\mathbf{u}\in \mathbb{Z}_{+}^{2}).  \label{moment0}
\end{equation}%
By a direct computation, we have
\[
\begin{tabular}{l}
$M_{\mathbf{u}}(k)\left( \Theta \left( W_{\omega }\right) \right) =$ \\
\\
$\left(
\begin{array}{ccccccc}
\gamma _{\mathbf{u}} & \gamma _{\mathbf{u}+\mathbf{\epsilon }_{1}} & \gamma
_{\mathbf{u}+\mathbf{\epsilon }_{2}} & \cdots  & \gamma _{\mathbf{u}+k%
\mathbf{\epsilon }_{1}} & \cdots  & \gamma _{\mathbf{u}+k\mathbf{\epsilon }%
_{2}} \\
\gamma _{\mathbf{u}+\mathbf{\epsilon }_{1}} & \gamma _{\mathbf{u}+2\mathbf{%
\epsilon }_{1}} & \gamma _{\mathbf{u}+\mathbf{\epsilon }_{1}+\mathbf{%
\epsilon }_{2}} & \cdots  & \gamma _{\mathbf{u}+(k+1)\mathbf{\epsilon }_{1}}
& \cdots  & \gamma _{\mathbf{u}+\mathbf{\epsilon }_{1}+k\mathbf{\epsilon }%
_{2}} \\
\gamma _{\mathbf{u}+\mathbf{\epsilon }_{2}} & \gamma _{\mathbf{u}+\mathbf{%
\epsilon }_{1}+\mathbf{\epsilon }_{2}} & \gamma _{\mathbf{u}+2\mathbf{%
\epsilon }_{2}} & \cdots  & \gamma _{\mathbf{u}+k\mathbf{\epsilon }_{1}+%
\mathbf{\epsilon }_{2}} & \cdots  & \gamma _{\mathbf{u}+(k+1)\mathbf{%
\epsilon }_{2}} \\
\vdots  & \vdots  & \vdots  & \ddots  & \vdots  & \ddots  & \vdots  \\
\gamma _{\mathbf{u}+k\mathbf{\epsilon }_{1}} & \gamma _{\mathbf{u}+(k+1)%
\mathbf{\epsilon }_{1}} & \gamma _{\mathbf{u}+k\mathbf{\epsilon }_{1}+%
\mathbf{\epsilon }_{2}} & \cdots  & \gamma _{\mathbf{u}+2k\mathbf{\epsilon }%
_{1}} & \cdots  & \gamma _{\mathbf{u}+k\mathbf{\epsilon }_{1}+k\mathbf{%
\epsilon }_{2}} \\
\vdots  & \vdots  & \vdots  & \ddots  & \vdots  & \ddots  & \vdots  \\
\gamma _{\mathbf{u}+k\mathbf{\epsilon }_{2}} & \gamma _{\mathbf{u}+\mathbf{%
\epsilon }_{1}+k\mathbf{\epsilon }_{2}} & \gamma _{\mathbf{u}+(k+1)\mathbf{%
\epsilon }_{2}} & \cdots  & \gamma _{\mathbf{u}+k\mathbf{\epsilon }_{1}+k%
\mathbf{\epsilon }_{2}} & \cdots  & \gamma _{\mathbf{u}+2k\mathbf{\epsilon }%
_{2}}%
\end{array}%
\right) ,$%
\end{tabular}%
\]%
which by (\ref{moment0}) equals
\[
\begin{tabular}{l}
$J_{\mathbf{u}}(k):=$ \\
\\
$\left(
\begin{array}{cccccc}
\gamma _{u_{1}+u_{2}} & \gamma _{u_{1}+u_{2}+1} & \cdots  & \gamma
_{u_{1}+u_{2}+k} & \cdots  & \gamma _{u_{1}+u_{2}+k} \\
\gamma _{u_{1}+u_{2}+1} & \gamma _{u_{1}+u_{2}+2} & \cdots  & \gamma
_{u_{1}+u_{2}+k+1} & \cdots  & \gamma _{u_{1}+u_{2}+k+1} \\
\vdots  & \vdots  & \ddots  & \vdots  & \ddots  & \vdots  \\
\gamma _{u_{1}+u_{2}+k} & \gamma _{u_{1}+u_{2}+k+1} & \cdots  & \gamma
_{u_{1}+u_{2}+2k} & \cdots  & \gamma _{u_{1}+u_{2}+2k} \\
\gamma _{u_{1}+u_{2}+k} & \gamma _{u_{1}+u_{2}+k+1} & \cdots  & \gamma
_{u_{1}+u_{2}+2k} & \cdots  & \gamma _{u_{1}+u_{2}+2k} \\
\vdots  & \vdots  & \ddots  & \vdots  & \ddots  & \vdots  \\
\gamma _{u_{1}+u_{2}+k} & \gamma _{u_{1}+u_{2}+u+1} & \ddots  & \gamma
_{u_{1}+u_{2}+2k} & \ddots  & \gamma _{u_{1}+u_{2}+2k} \\
\gamma _{u_{1}+u_{2}+k} & \gamma _{u_{1}+u_{2}+k+1} & \cdots  & \gamma
_{u_{1}+u_{2}+2k} & \cdots  & \gamma _{u_{1}+u_{2}+2k}%
\end{array}%
\right) .$%
\end{tabular}%
\]%
For $1\leq i\leq k+1$, we can observe that the%
\[
\left( \frac{i(i+1)}{2}+1\right) ^{th},\left( \frac{i(i+1)}{2}+2\right)
^{th},\cdots ,\left( \frac{i(i+1)}{2}+(i+1)\right) ^{th}
\]%
rows and columns of $J_{\mathbf{u}}(k)$ are equal. \ Thus, a direct
calculation (i.e., discarding some redundant rows and columns in the matrix $%
J_{\mathbf{u}}(k)$) shows that
\begin{equation}
\begin{tabular}{l}
$J_{\mathbf{u}}(k)\geq 0\iff L_{\mathbf{u}}(k)\geq 0$,%
\end{tabular}
\label{condition02}
\end{equation}%
where%
\[
L_{\mathbf{u}}(k):=\left(
\begin{array}{cccc}
\gamma _{u_{1}+u_{2}}\left( W_{\omega }\right)  & \gamma
_{u_{1}+u_{2}+1}\left( W_{\omega }\right)  & \cdots  & \gamma
_{u_{1}+u_{2}+k}\left( W_{\omega }\right)  \\
\gamma _{u_{1}+u_{2}+1}\left( W_{\omega }\right)  & \gamma
_{u_{1}+u_{2}+2}\left( W_{\omega }\right)  & \cdots  & \gamma
_{u_{1}+u_{2}+k+1}\left( W_{\omega }\right)  \\
\vdots  & \vdots  & \ddots  & \vdots  \\
\gamma _{u_{1}+u_{2}+k}\left( W_{\omega }\right)  & \gamma
_{u_{1}+u_{2}+k+1}\left( W_{\omega }\right)  & \cdots  & \gamma
_{u_{1}+u_{2}+2k}\left( W_{\omega }\right)
\end{array}%
\right) .
\]

Note that%
\begin{equation}
L_{\left( u_{1},u_{2}\right) }(k)\geq 0\Longleftrightarrow
H(k;u_{1}+u_{2})\left( W_{\omega }\right) \geq 0\text{.}  \label{condition03}
\end{equation}%
Thus, if $W_{\omega }$ is $k$-hyponormal then $H(k;u)\left(
W_{\omega }\right) \ge 0$ for all $(u\geq 0)$, which a fortiori implies that $M_{\mathbf{u}}(k)\left(
W_{(\alpha ,\beta )}\right) \geq 0$ for all $\mathbf{u}\in \mathbb{Z}%
_{+}^{2} $, as desired. \ The proof is now complete.
\end{proof}

Now we have the following result.

\begin{theorem} 
\label{pre-hypo}Consider the $2$-variable weighted shift $\Theta \left( W_{\omega }\right) \equiv \mathbf{(}%
T_{1},T_{2})$ given by Figure \ref{Figure 6}(i). \ Suppose that $\Theta
\left( W_{\omega }\right) $ is hyponormal. \ Then, the toral Aluthge transform $%
\widetilde{\Theta \left( W_{\omega }\right)} \equiv \Theta \left( \widetilde{%
W}_{\omega }\right) $ is also hyponormal.
\end{theorem}

In view of Lemma \ref{lem51}, we immediately get

\begin{corollary} \ The conclusion of Theorem \ref{pre-hypo} holds in the class $\mathcal{A}_{\mathcal{TS}}$.
\end{corollary}

\begin{proof}[Proof of Theorem \ref{pre-hypo}] \ 
Since $\Theta \left( W_{\omega }\right) $ is hyponormal, by Proposition \ref%
{propscaling2}, $W_{\omega }$ is hyponormal. \ Thus, we have that for any
integer $n\geq 0$, $\omega _{n}\leq \omega _{n+1}\Longrightarrow \sqrt{%
\omega _{n}\omega _{n+1}}\leq \sqrt{\omega _{n+1}\omega _{n+2}}$, which
implies that $\widetilde{W}_{\omega }$ is also hyponormal. \ By
Proposition \ref{propscaling2}, $\widetilde{\Theta \left( W_{\omega
}\right)} $ is hyponormal, as desired.
\end{proof}

\begin{remark}
(i) \ We construct an example $\Theta \left( W_{\omega }\right) $ such that $%
\Theta \left( W_{\omega }\right) $ is not hyponormal, but the Aluthge
transform $\widetilde{\Theta \left( W_{\omega }\right)} $ of $\Theta \left(
W_{\omega }\right) $ is hyponormal. \ Consider the unilateral weighted shift introduced in Section \ref{Int}, that is, $W_{\omega }\equiv \mathrm{%
shift}\left( \frac{1}{2},2,\frac{1}{2},2,\frac{1}{2},2,\cdots \right)$. \ $W_{\omega }$ is not hyponormal, but the
Aluthge transform $\widetilde{W}_{\omega }=U_{+}$ is subnormal. \ Thus, by
Proposition \ref{propscaling2}, we have that $\Theta \left( W_{\omega
}\right) $ is not hyponormal, but $\widetilde{\Theta \left( W_{\omega
}\right)} $ is hyponormal, as desired. \newline
(ii) \ Using an argument entirely similar to that in (i) above, one can show that $2$-hyponor-mality is not preserved by the toral or spherical Aluthge transform (as in the single variable case). \qed
\end{remark}

We can easily observe that if $W_{(\alpha ,\beta )}$ is of tensor form, that is, $(I\otimes W_{\sigma
},W_{\tau }\otimes I)$, then its toral Aluthge transform $\widetilde{W}_{(\alpha ,\beta )}$ is also of tensor form; however, the spherical
Aluthge transform $\widetilde{W}_{(\alpha ,\beta )}$ is in general not of tensor form. \ In any event, hyponormality is invariant under both
Aluthge transforms when $W_{(\alpha ,\beta )}$ is of tensor form. \ That the toral Aluthge transform preserves hyponormality for these $2$-variable weighted shifts is clear; we now establish invariance of hyponormality for the spherical Aluthge transform. \ Recall first that, by Proposition \ref{Proposition1}, the spherical Aluthge transform is commuting.  

\begin{proposition}
\label{tensor} Let $W_{(\alpha ,\beta )}$ be a $2$%
-variable weighted shift of \textit{tensor form} $(I\otimes W_{\sigma
},W_{\tau }\otimes I)$, and assume that $W_{\sigma}$ and $W_{\tau}$ are hyponormal. \ Then $\widehat{W}_{(\alpha ,\beta )}$ is hyponormal.
\end{proposition}

\begin{proof}
Without loss of generality, we can assume that 
$$
W_{\sigma }\equiv shift(\sqrt{x},\sqrt{y},1,\cdots ) \; \; \textrm{ and } \; \; W_{\tau }\equiv shift(\sqrt{a},\sqrt{b},1,\cdots),
$$
with $0<x<y<1$ and $0<a<b<1$. \ Also, it is enough to focus on the Six-Point Test at $(0,0)$ (cf. \cite[Theorem 6.1]{bridge}, \cite[Theorem 1.3]{CuYo1}); that is, we will check that $M_{\left(
0,0\right) }(1)\left( \widehat{W}_{(\alpha ,\beta )}\right) \geq 0$. 

Observe that 
\begin{equation*}
M_{\left( 0,0\right) }(1)\left( \widehat{W}_{(\alpha ,\beta )}\right)
=\left(
\begin{array}{ccc}
1 & x\sqrt{\frac{a+y}{a+x}} & a\sqrt{\frac{b+x}{a+x}} \\
x\sqrt{\frac{a+y}{a+x}} & xy\sqrt{\frac{a+1}{a+x}} & ax\sqrt{\frac{b+y}{a+x}}
\\
a\sqrt{\frac{b+x}{a+x}} & ax\sqrt{\frac{b+y}{a+x}} & ab\sqrt{\frac{x+1}{a+x}}%
\end{array}%
\right) \text{.}
\end{equation*}%
Thus, we obtain%
\begin{equation*}
M_{\left( 0,0\right) }(1)\left( \widehat{W}_{(\alpha ,\beta )}\right) \geq
0\Longleftrightarrow A\geq 0\text{,}
\end{equation*}%
where%
\begin{equation*}
A:=\left(
\begin{array}{ccc}
\sqrt{a+x} & x\sqrt{a+y} & a\sqrt{b+x} \\
x\sqrt{a+y} & xy\sqrt{a+1} & ax\sqrt{b+y} \\
a\sqrt{b+x} & ax\sqrt{b+y} & ab\sqrt{x+1}%
\end{array}%
\right) \text{.}
\end{equation*}%
Now modify the $(2,2)$ and $(3,3)$ entries of $A$ and let 
\begin{equation*}
B:=\left(
\begin{array}{ccc}
\sqrt{a+x} & x\sqrt{a+y} & a\sqrt{b+x} \\
x\sqrt{a+y} & xy\sqrt{a+x} & ax\sqrt{b+y} \\
a\sqrt{b+x} & ax\sqrt{b+y} & ab\sqrt{a+x}%
\end{array}%
\right) \text{.}
\end{equation*}%
A direct calculation using the Nested Determinant Test shows that $A\geq B$ and that $B\geq 0$. \ Thus, we have
\begin{equation*}
M_{\left( 0,0\right) }(1)\left( \widehat{W}_{(\alpha ,\beta )}\right) \geq 0%
\text{,}
\end{equation*}
so that $\widehat{W}_{(\alpha ,\beta )}$ is hyponormal, as desired.
\end{proof}

\begin{remark} \ One might be tempted to claim that subnormality is also preserved by the toral and spherical Aluthge transforms, within the class of $2$-variable weighted shifts of tensor form. \ However, this is not the case. \ Indeed, in \cite{Ex} and \cite{Exn}, G. Exner considered the weighted shift $W_{\sigma}$ with $3$-atomic Berger measure $\frac{1}{3}(\delta_{0}+\delta_{\frac{1}{2}}+\delta_{1})$ (studied in \cite{CPY}) and proved that the Aluthge transform of $W_{\sigma}$ is not subnormal.
\end{remark}


\section{\label{Sect4}Continuity Properties of the Aluthge Transforms}

We turn our attention to the continuity properties of the Aluthge transforms
of a commuting pair. \ The following result is well known. \ For a single
operator $T\in \mathcal{B}(\mathcal{H})$, the Aluthge transform map $%
T\rightarrow \widetilde{T}$ is $\left( \left\Vert \cdot \right\Vert
,\left\Vert \cdot \right\Vert \right) -$ continuous on $\mathcal{B}(\mathcal{%
H})$ (\cite{DySc}). \ We want to extend the result to multivariable case. \
First, we define the operator norm of $\mathbf{T}\equiv (T_1,T_2)$ as
\begin{equation}
\left\Vert \mathbf{T}\right\Vert :=\max \left\{ \left\Vert T_1 \right\Vert
,\left\Vert T_2 \right\Vert \right\}.  \label{normdef}
\end{equation}

\begin{theorem}
\label{ContinuityC}The toral Aluthge transform map $\mathbf{T} \rightarrow
\widetilde{\mathbf{T}}$ is $\left( \left\Vert \cdot \right\Vert ,\left\Vert \cdot \right\Vert \right) -$ continuous on $\mathcal{%
B}(\mathcal{H})$.
\end{theorem}

\begin{proof}
Straightforward from the definition of $\left\Vert \mathbf{T}\right\Vert$.

\end{proof}

We turn our attention to the continuity properties in $\mathcal{B}(\mathcal{H%
})$ for the spherical Aluthge transform of a commuting pair. \ For this, we need
a couple of auxiliary results, which can be proved by suitable adaptations of the results in \cite[Lemmas 2.1 and 2.2]{DySc}.

\begin{lemma}
\label{Re 4} \ Let $\mathbf{T}\equiv (T_{1},T_{2}) \equiv (V_1P,V_2P)$ be a pair of commuting operators, written in joint polar decomposition form, where $P=\sqrt{T_{1}^{\ast }T_{1}+T_{2}^{\ast }T_{2}}$. \ For $n\in \mathbb{N}$ and $t>0$, let $f_{n}\left(
t\right) :=\sqrt{\max \left( \frac{1}{n},t\right) }$ and let $A_n:=f_n(\mathbf{T})$. \ Then:
\newline
(i) \ $\left\Vert A_{n} \right\Vert \leq \max \left( n^{-%
\frac{1}{2}},\left\Vert P\right\Vert ^{\frac{1}{2}}\right) $;\newline
(ii) \ $\left\Vert P A_n^{-1}\right\Vert
\leq \left\Vert P\right\Vert ^{\frac{1}{2}}$;\newline
(iii) \ $\left\Vert A_n -P^{\frac{1}{2}}\right\Vert \leq
n^{-\frac{1}{2}}$;\newline
(iv) \ $\left\Vert PA_n^{-1}-P^{\frac{1}{2}%
}\right\Vert \leq \frac{1}{4} n^{-\frac{1}{2}}$;\newline
(v) \ For $i=1,2$, $\left\Vert A_n T_{i}A_n^{-1}-P^{\frac{1}{2}}V_{i} P^{\frac{1}{2}}\right\Vert \leq
\frac{5}{4} n^{-\frac{1}{2}}\left\Vert T_{i}\right\Vert ^{\frac{1}{2}}$.
\end{lemma}

\begin{lemma}
\label{Re 5} \ Given $R\geq 1$ and $\epsilon >0$, there are real polynomials $p$ and $q$ such that for every commuting pair $\mathbf{T}%
\equiv (T_{1},T_{2})$ with $\left\Vert T_{i}\right\Vert
\leq R$ $\left( i=1,2\right) $, we have
\begin{equation*}
\left\Vert P^{\frac{1}{2}}V_{i}P^{\frac{1}{2}}-p\left( T_{1}^{\ast
}T_{1}+T_{2}^{\ast }T_{2}\right) T_{i}q\left( T_{1}^{\ast }T_{1}+T_{2}^{\ast
}T_{2}\right) \right\Vert <\epsilon \text{.}
\end{equation*}
\end{lemma}

In the statement below, $\left\Vert \cdot \right\Vert $ refers to the
operator norm topology on $\mathcal{B}(\mathcal{H})^2$ (see \ref{normdef}).

\begin{theorem}
\label{ContinuityP}The spherical Aluthge transform 
$$
\mathbf{(}%
T_{1},T_{2})\rightarrow \widehat{\mathbf{(}T_{1},T_{2})}
$$
is $\left(
\left\Vert \cdot \right\Vert ,\left\Vert \cdot \right\Vert \right) -$
continuous on $\mathcal{B}(\mathcal{H})$.
\end{theorem}

\begin{proof}
Observe first that $\left\Vert V_i \right\Vert \le 1$ for $i=1,2$, as follows from the inequality
$$
\left\Vert V_iPx \right\Vert^2 =\left\Vert T_ix \right\Vert^2 = <T_i^*T_ix,x> \le <P^2x,x> = \left\Vert Px \right\Vert^2.
$$
The proof is now an easy consequence of the Proof of \cite[Theorem 2.3]{DySc}, when one uses Lemma \ref{Re 5} instead of \cite[Lemma 2.2]{DySc}.
\end{proof}


\section{\label{Sect5}Spectral Properties of the Aluthge Transforms}

In this section, we study whether the multivariable Aluthge transforms preserve the Taylor spectrum and Taylor essential spectrum, when $%
W_{(\alpha ,\beta )}$ is in the class $\mathcal{TC}$ of $2$-variable weighted shifts with core of tensor form; this is a large nontrivial class, which has been previously studied in [15--20], [24--28] and [47--49].


We begin by looking at the toral Aluthge transform. \ By Proposition \ref%
{commuting1} and Remark \ref{Re 1}, we note that the weight diagram of $%
W_{(\alpha ,\beta )}$ is as in Figure \ref{Figure 2ex}(i), provided the toral Aluthge transform is commutative. \ We first address the Taylor spectrum.

\begin{lemma}
\label{lem1} (i) (\cite{Cu1}, \cite{Cu3}) \ Let $\mathcal{H}_{1}$ and $%
\mathcal{H}_{2}$ be Hilbert spaces, and let $A_{i}\in \mathcal{B}(\mathcal{H}%
_{1}),$ $C_{i}\in \mathcal{B}(\mathcal{H}_{2})$ and $B_{i}\in \mathcal{B}(%
\mathcal{H}_{1},\mathcal{H}_{2}),(i=1,\cdots ,n)$ be such that%
\begin{equation*}
\left(
\begin{array}{cc}
\mathbf{A} & 0 \\
\mathbf{B} & \mathbf{C}%
\end{array}%
\right) :=\left( \left(
\begin{array}{cc}
A_{1} & 0 \\
B_{1} & C_{1}%
\end{array}%
\right) ,\ldots ,\left(
\begin{array}{cc}
A_{n} & 0 \\
B_{n} & C_{n}%
\end{array}%
\right) \right)
\end{equation*}%
is commuting. \ Assume that $\mathbf{A}$ and $\left(
\begin{array}{cc}
\mathbf{A} & 0 \\
\mathbf{B} & \mathbf{C}%
\end{array}%
\right) $ are Taylor invertible. \ Then, $\mathbf{C}$ is Taylor invertible.
\ Furthermore, if $\mathbf{A}$ and $\mathbf{C}$ are Taylor invertible, then $%
\left(
\begin{array}{cc}
\mathbf{A} & 0 \\
\mathbf{B} & \mathbf{C}%
\end{array}%
\right) $ is Taylor invertible.\newline
(ii) (\cite{CuFi}) \ For $\mathbf{A}$ and $\mathbf{B}$ two commuting $n$-tuples of bounded
operators on Hilbert space, we have:%
\begin{equation*}
\sigma _{T}(I \otimes \mathbf{A},\mathbf{B} \otimes I)=\sigma _{T}(\mathbf{A}%
)\times \sigma _{T}(\mathbf{B})
\end{equation*}%
and%
\begin{equation*}
\sigma _{Te}(I \otimes \mathbf{A}, \mathbf{B} \otimes I)=\sigma _{Te}\left(
\mathbf{A}\right) \times \sigma _{T}\left( \mathbf{B}\right) \cup \sigma
_{T}\left( \mathbf{A}\right) \times \sigma _{Te}\left( \mathbf{B}\right)
\text{.}
\end{equation*}
\end{lemma}
\ To apply Lemma \ref{lem1}, we first let%
\begin{equation*}
\begin{tabular}{l}
$W_{\omega }:=\mathrm{shift}\left( \omega _{0},\omega _{1},\cdots \right) $,
$W_{\tau }:=\mathrm{shift}\left( \tau _{0},\tau _{1},\cdots \right) $,  \\
$W_{\omega ^{(0)}}:=\mathrm{shift}\left( x_{0},x_{1},\omega _{1},\omega
_{2},\cdots \right) $, $W_{\omega ^{(a)}}:=\mathrm{shift}\left( a,\omega
_{0},\omega _{1},\cdots \right) $,  \\
$W_{\omega ^{(2)}}:=\mathrm{shift}\left( a\frac{\omega _{0}}{x_{1}},\omega
_{0},\omega _{1},\cdots \right) $, $I:=\mathrm{diag}\left( 1,1,\cdots
\right) $, \\
$D_{1}:=\mathrm{diag}\left( y_{0},a\frac{y_{0}}{x_{0}},a\frac{\omega
_{0}y_{0}}{x_{0}x_{1}},a\frac{\omega _{0}y_{0}}{x_{0}x_{1}},\cdots \right) $%
, and $D_{2}:=\mathrm{diag}\left( \frac{x_{1}\tau _{0}}{\omega _{0}},\tau
_{0},\tau _{0},\cdots \right) $, \\
where $I$ is the identity matrix.%
\end{tabular}
\end{equation*}

\begin{theorem}
\label{thmtaylor}Consider a commuting $2$-variable weighted shift $W_{(\alpha ,\beta
)}\equiv (T_{1},T_{2})$ with weight diagram given by Figure \ref{Figure 2ex}(i). \ Assume also that $T_1$ and $T_2$ are hyponormal. \ Then 
\begin{equation}
\begin{tabular}{l}
\begin{tabular}{l}
$\sigma _{T}\left( W_{(\alpha ,\beta )}\right) =\left( \left\Vert W_{\omega
}\right\Vert \cdot \overline{\mathbb{D}}\times \left\Vert W_{\tau
}\right\Vert \cdot \overline{\mathbb{D}}\right) $ and \\
$\sigma _{Te}\left( W_{(\alpha ,\beta )}\right) =\left( \left\Vert W_{\omega
}\right\Vert \cdot \mathbb{T}\times \left\Vert W_{\tau }\right\Vert \cdot
\overline{\mathbb{D}}\right) \cup \left( \left\Vert W_{\omega }\right\Vert
\cdot \overline{\mathbb{D}}\times \left\Vert W_{\tau }\right\Vert \cdot
\mathbb{T}\right) .$%
\end{tabular}%
\end{tabular}
\label{1}
\end{equation}%
Here $\overline{\mathbb{D}}$ denotes the closure of the open unit disk $%
\mathbb{D}$ and $\mathbb{T}$ the unit circle.
\end{theorem}

\begin{proof}
We represent $W_{(\alpha ,\beta )}\equiv (T_{1},T_{2})$ by block matrices
relative to the decomposition
\begin{equation}
\ell ^{2}(\mathbb{Z}_{+}^{2})\cong \ell ^{2}(\mathbb{Z}_{+})\otimes \ell
^{2}(\mathbb{Z}_{+})\cong (\ell ^{2}(\mathbb{Z}_{+})\oplus \ell ^{2}(%
\mathbb{Z}_{+}))\oplus (\ell ^{2}(\mathbb{Z}_{+})\oplus \cdots) \text{.}  \label{decom}
\end{equation}%
\ Then, we obtain%
\begin{equation*}
T_{1}\equiv \left(
\begin{array}{cccc}
W_{\omega ^{(0)}} &  & \vdots  &  \\
& W_{\omega ^{(1)}} & \vdots  &  \\
\cdots  & \cdots  &  & \cdots  \\
&  & \vdots  & R_{1}%
\end{array}%
\right) \text{ and }T_{2}\equiv \left(
\begin{array}{cccc}
0 &  & \vdots  &  \\
D_{1} & 0 & \vdots  &  \\
\cdots  & \cdots  &  & \cdots  \\
& D_{2} & \vdots  & R_{2}%
\end{array}%
\right) \text{,}
\end{equation*}%
where $R_{1}:=\left(
\begin{array}{ccc}
W_{\omega ^{(2)}} &  &  \\
& W_{\omega ^{(2)}} &  \\
&  & \ddots
\end{array}%
\right) $ and $R_{2}:=\left(
\begin{array}{cccc}
0 &  &  &  \\
\tau _{1}I & 0 &  &  \\
& \tau _{2}I & 0 &  \\
&  & \ddots  & \ddots
\end{array}%
\right) $. \newline 
We first consider $\sigma _{T}(W_{(\alpha ,\beta )})$ of $%
W_{(\alpha ,\beta )}\equiv (T_{1},T_{2})$. \ Note the following:%
\begin{equation}
\begin{tabular}{l}
$\left\Vert W_{\omega }\right\Vert =\left\Vert W_{\omega ^{(0)}}\right\Vert
=\left\Vert W_{\omega ^{(1)}}\right\Vert =\left\Vert W_{\omega
^{(2)}}\right\Vert \text{ and}$ \\
$\left\Vert W_{\tau }\right\Vert =\left\Vert \mathrm{shift}\left( \tau
_{1},\tau _{2},\cdots \right) \right\Vert =\left\Vert \mathrm{shift}\left(
\frac{\omega _{0}y_{0}}{x_{0}x_{1}},\tau _{0},\tau _{1},\cdots \right)
\right\Vert \text{.}$%
\end{tabular}
\label{norm}
\end{equation}%
Thus, by Lemma \ref{lem1}(i) and (\ref{norm}), we have%
\begin{equation}
\begin{tabular}{l}
$\sigma _{T}(T_{1},T_{2})$ \\
$\subseteq \sigma _{T}\left( \left(
\begin{array}{cc}
W_{\omega ^{(0)}} & 0 \\
0 & W_{\omega ^{(1)}}%
\end{array}%
\right) ,\left(
\begin{array}{cc}
0 & 0 \\
D_{1} & 0%
\end{array}%
\right) \right) \cup \sigma _{T}\left( I \otimes W_{\omega ^{(2)}}, W_{\tau} \otimes I \right) $ \\
$\subseteq \sigma _{T}\left( W_{\omega ^{(0)}},0\right) \cup \sigma
_{T}\left( W_{\omega ^{(1)}},0\right) \cup \sigma _{T}\left( I \otimes W_{\omega
^{(2)}}, W_{\tau} \otimes I \right) $ \\
$=\left( \left\Vert W_{\omega }\right\Vert \cdot \overline{\mathbb{D}}\times
\{0\}\right) \cup \left( \left\Vert W_{\omega }\right\Vert \cdot \overline{%
\mathbb{D}}\times \left\Vert W_{\tau }\right\Vert \cdot \overline{\mathbb{D}}%
\right) =\left\Vert W_{\omega }\right\Vert \cdot \overline{\mathbb{D}}\times
\left\Vert W_{\tau }\right\Vert \cdot \overline{\mathbb{D}}$%
\end{tabular}
\label{2}
\end{equation}%
By Lemma \ref{lem1}(ii) and (\ref{norm}), we have%
\begin{equation}
\begin{tabular}{l}
$\sigma _{T}\left( I \otimes W_{\omega ^{(2)}}, W_{\tau} \otimes I \right) $ \\
$\subseteq \sigma _{T}\left( \left(
\begin{array}{cc}
W_{\omega ^{(0)}} & 0 \\
0 & W_{\omega ^{(1)}}%
\end{array}%
\right) ,\left(
\begin{array}{cc}
0 & 0 \\
D_{1} & 0%
\end{array}%
\right) \right) \cup \sigma _{T}(T_{1},T_{2})$ \\
$\Rightarrow \sigma _{T}\left( I \otimes W_{\omega ^{(2)}}, W_{\tau} \otimes I \right) $ \\
$\ \ \ \subseteq \sigma _{T}\left( W_{\omega ^{(0)}},0\right) \cup \sigma
_{T}\left( W_{\omega ^{(1)}},0\right) \cup \sigma _{T}(T_{1},T_{2})$ \\
$\Rightarrow \left\Vert W_{\omega }\right\Vert \cdot \overline{\mathbb{D}}%
\times \left\Vert W_{\tau }\right\Vert \cdot \overline{\mathbb{D}}\subseteq
\left\Vert W_{\omega }\right\Vert \cdot \overline{\mathbb{D}}\times
\{0\}\cup \sigma _{T}(T_{1},T_{2})$ \\
$\Rightarrow \left( \left\Vert W_{\omega }\right\Vert \cdot \overline{%
\mathbb{D}}\times \left\Vert W_{\tau }\right\Vert \cdot \overline{\mathbb{D}}%
\right) $ $\backslash $ $\left( \left\Vert W_{\omega }\right\Vert \cdot
\overline{\mathbb{D}}\times \{0\}\right) \subseteq \sigma _{T}(T_{1},T_{2})$.%
\end{tabular}
\label{3}
\end{equation}%
Since the Taylor spectrum $\sigma _{T}(T_{1},T_{2})$ is a closed set in $\mathbb{%
C}\times \mathbb{C}$, by (\ref{equa1}) and (\ref{equa2}), we can get%
\begin{equation}
\sigma _{T}\left( W_{(\alpha ,\beta )}\right) =\left( \left\Vert W_{\omega
}\right\Vert \cdot \overline{\mathbb{D}}\times \left\Vert W_{\tau
}\right\Vert \cdot \overline{\mathbb{D}}\right) \text{.}  \label{4}
\end{equation}%
We next consider the Taylor essential spectrum $\sigma _{Te}(T_{1},T_{2})$
of $W_{(\alpha ,\beta )}\equiv (T_{1},T_{2})$. \ Observe that $W_{\omega
^{(2)}}$ is a compact perturbation of $W_{\omega ^{(1)}}$ and $W_{\omega
^{(0)}}$. \ Also, $\frac{\omega _{0}y_{0}}{x_{0}x_{1}}I$ and $\tau _{0}I$ are compact perturbations of $D_{1}$ and $D_{2}$, respectively. \ Thus, we have%
\begin{equation}
\begin{tabular}{l}
$\sigma _{Te}(T_{1},T_{2})=\sigma _{Te}\left( I \otimes W_{\omega ^{(2)}}
, \mathrm{shift}\left( \frac{\omega _{0}y_{0}}{x_{0}x_{1}},\tau
_{0},\tau _{1},\cdots \right) \otimes I \right) $.%
\end{tabular}
\label{5}
\end{equation}%
By Lemma \ref{lem1}(ii) and (\ref{norm}), we note%
\begin{equation}
\begin{tabular}{l}
$\sigma _{Te}\left( I \otimes W_{\omega ^{(2)}}, \mathrm{shift}\left(
\frac{\omega _{0}y_{0}}{x_{0}x_{1}},\tau _{0},\tau _{1},\cdots \right) \otimes I
\right) $ \\
$=\sigma _{Te}\left( W_{\omega ^{(2)}}\right) \times \sigma _{T}\left(
W_{\tau }\right) \cup \sigma _{T}\left( W_{\omega ^{(2)}}\right) \times
\sigma _{Te}\left( W_{\tau }\right) $ \\
$=\left( \left\Vert W_{\omega }\right\Vert \cdot \mathbb{T}\times \left\Vert
W_{\tau }\right\Vert \cdot \overline{\mathbb{D}}\right) \cup \left(
\left\Vert W_{\omega }\right\Vert \cdot \overline{\mathbb{D}}\times
\left\Vert W_{\tau }\right\Vert \cdot \mathbb{T}\right) $.%
\end{tabular}
\label{6}
\end{equation}%
Therefore, our proof is now complete.
\end{proof}

\begin{theorem}
\label{thmtaylor-Alu}(Case of toral Aluthge Transform) \ Consider a commuting $2$-variable weighted shift $W_{(\alpha
,\beta )}\equiv (T_{1},T_{2})$ with weight diagram given by Figure \ref{Figure 2ex}(i). \ Assume also that $T_1$ and $T_2$ are hyponormal. \ Then
$$
\sigma _{T}\left( \widetilde{W}_{(\alpha ,\beta )}\right) =\left(
\left\Vert W_{\omega }\right\Vert \cdot \overline{\mathbb{D}}\times
\left\Vert W_{\tau }\right\Vert \cdot \overline{\mathbb{D}}\right)
$$
and
$$
\sigma _{Te}\left( \widetilde{W}_{(\alpha ,\beta )}\right) = \left(
\left\Vert W_{\omega }\right\Vert \cdot \mathbb{T}\times \left\Vert W_{\tau
}\right\Vert \cdot \overline{\mathbb{D}}\right) \cup \left( \left\Vert
W_{\omega }\right\Vert \cdot \overline{\mathbb{D}}\times \left\Vert W_{\tau
}\right\Vert \cdot \mathbb{T}\right) .
$$
\end{theorem}

\begin{proof}
Since the structure of the weight diagram for $\widetilde{W}_{(\alpha ,\beta )}\equiv
\left( \widetilde{T}_{1},\widetilde{T}_{2}\right) $ is entirely similar to that of $W_{(\alpha
,\beta )}\equiv (T_{1},T_{2})$, the results follows by imitating the Proof of Theorem \ref{thmtaylor}.
\end{proof}

By the results of Theorems \ref{thmtaylor} and \ref{thmtaylor-Alu}, we easily obtain the following result.

\begin{corollary}
\label{Cor1}Consider a commuting $2$-variable weighted shift $W_{(\alpha ,\beta
)}\equiv (T_{1},T_{2})$ with weight diagram given by Figure \ref{Figure 2ex}(i). \ Assume also that $T_1$ and $T_2$ are hyponormal. \ Then
$$
\sigma _{T}\left( \widetilde{W}_{(\alpha ,\beta )}\right) =\sigma _{T}\left( {W}_{(\alpha ,\beta )}\right)
$$
and
$$
\sigma _{Te}\left( \widetilde{W}_{(\alpha ,\beta )}\right) =\sigma _{Te}\left({W}_{(\alpha ,\beta )}\right).
$$
\end{corollary}

\begin{remark}
\label{Re 3} We note that the commutativity property is required to
check the Taylor spectrum (resp. Taylor essential spectrum) of $\widetilde{W}%
_{(\alpha ,\beta )}$. \ Thus, if $W_{(\alpha ,\beta )}\in \mathcal{TC}$, then $\widetilde{W}_{(\alpha ,\beta )}$ is commuting. \newline
(ii) \ By Corollary \ref{Cor1}, we can see that the Taylor spectrum (resp.
Taylor essential spectrum) of $\widetilde{W}_{(\alpha ,\beta )}$ equals that
of $W_{(\alpha ,\beta )}$ when $W_{(\alpha ,\beta )}$ is commuting and $T_1$ and $T_2$ are hyponormal. \
Thus, Corollary \ref{Cor1} gives a partial solution to Problem \ref{problem
3}.
\end{remark}

We now turn our attention to the case of the spherical Aluthge Transform. \ We need a preliminary result.

\begin{proposition} \label{proptensor}
Let $W_{(\alpha ,\beta )} \in \mathcal{TC}$, with core $c(W_{(\alpha ,\beta )})=(I \otimes W_{\sigma}, W_{\tau} \otimes I)$. \ Then $\widehat{W}_{(\alpha
,\beta )}\in \mathcal{TC}$ if and only $c(W_{(\alpha ,\beta )})=(rI \otimes U_+, W_{\tau} \otimes I)$ or 
$c(W_{(\alpha ,\beta )})=(I \otimes W_{\sigma}, U_+ \otimes sI)$ for some $r,s>0$.
\end{proposition}

\begin{proof}
By Lemma \ref{PolarAlu}, we recall that
\begin{equation}
\widehat{T}_{1}e_{\mathbf{k}}=\alpha _{\mathbf{k}}\frac{(\alpha _{\mathbf{k}+%
\mathbf{\epsilon }_{1}}^{2}+\beta _{\mathbf{k}+\mathbf{\epsilon }%
_{1}}^{2})^{1/4}}{(\alpha _{\mathbf{k}}^{2}+\beta _{\mathbf{k}}^{2})^{1/4}}%
e_{\mathbf{k}+\mathbf{\epsilon }_{1}}\text{ and }\widehat{T}_{2}e_{\mathbf{k}%
}=\beta _{\mathbf{k}}\frac{(\alpha _{\mathbf{k}+\mathbf{\epsilon }%
_{2}}^{2}+\beta _{\mathbf{k}+\mathbf{\epsilon }_{2}}^{2})^{1/4}}{(\alpha _{%
\mathbf{k}}^{2}+\beta _{\mathbf{k}}^{2})^{1/4}}e_{\mathbf{k}+\mathbf{%
\epsilon }_{2}}  \label{Weight of Polar}
\end{equation}

Since $W_{(\alpha ,\beta )}\in \mathcal{TC}$, and since $\widehat{W}_{(\alpha ,\beta )}$ leaves the subspace $\mathcal{M} \cap \mathcal{N}$ invariant, it readily follows that the spherical Aluthge transform of $c(W_{(\alpha ,\beta )})$ is $c(\widehat{W}_{(\alpha ,\beta )})$. \ As a result, we may assume, without loss of generality, that $W_{(\alpha ,\beta )}\equiv (I \otimes W_{\sigma },W_{\tau
}\otimes I)$, where $W_{\sigma }$ and $W_{\tau }$ are unilateral weighted shifts. \ Now, by (\ref{Weight of Polar}), for all $\mathbf{k}%
\equiv (k_{1},k_{2})\in \mathbb{Z}_{+}^{2}$ we have
\begin{eqnarray*}
\widehat{T}_{1}e_{\mathbf{k}} &=&\widehat{T}_{1}e_{\mathbf{k}+\mathbf{%
\epsilon }_{2}}\Longleftrightarrow \frac{\sigma _{k_{1}+1}^{2}+\tau
_{k_{2}}^{2}}{\sigma _{k_{1}}^{2}+\tau _{k_{2}}^{2}}=\frac{\sigma _{k_{1}+%
1}^{2}+\tau _{k_{2}+1}^{2}}{\sigma _{k_{1}}^{2}+\tau
_{k_{2}+1}^{2}} \Longleftrightarrow  \widehat{T}_{2}e_{\mathbf{k}} =\widehat{T}_{1}e_{\mathbf{k}+\mathbf{\epsilon }_{1}}\\
&\Longleftrightarrow &\left( \tau _{k_{2}+1}^{2}-\tau
_{k_{2}}^{2}\right) \left(\sigma _{k_{1}}^{2} - \sigma_{k_{1}+1}^{2} \right) 
\Longleftrightarrow \tau_{k_{2}+1}=\tau _{k_{2}} \; \; \textrm{or } \; \; \sigma _{k_{1}+1}=\sigma
_{k_{1}}.
\end{eqnarray*}
If there exists $k_2 \in \mathbb{Z}_+$ such that $\tau_{k_{2}+1} \ne \tau _{k_{2}}$, then for all $k_1 \in \mathbb{Z}_+$ we must have $\sigma_{k_1+1}=\sigma_{k_1}$; that is, $\sigma_{k_1}=\sigma_0$ for all $k_1$. \ On the other hand, if $\tau_{k_{2}+1}=\tau _{k_{2}}$ for all $k_2 \in \mathbb{Z}_+$, then $\tau_{k_2}=\tau_0$ for all $k_2$. \ This completes the proof.
\end{proof}

We are now ready to state
\begin{theorem} (Case of spherical Aluthge Transform) \label{thm87}
Let $W_{(\alpha ,\beta )}\equiv \left( T_{1},T_{2}\right) \in \mathcal{TC}$
be as in Proposition 8.6. \ Assume also that $T_1$ and $T_2$ are hyponormal. \ Then
\begin{equation*}
\sigma _{T}\left(\widehat{W}_{(\alpha ,\beta )}\right) =\sigma _{T}\left(W_{(\alpha ,\beta )}\right) \text{ and }\sigma _{T_{e}}\left(\widehat{W}_{(\alpha
,\beta )}\right) =\sigma _{T_{e}}\left(W_{(\alpha ,\beta)}\right) \text{.}
\end{equation*}
\end{theorem}

\begin{proof}
By Proposition \ref{proptensor}, without loss of generality we may assume 
$$
c\left( W_{(\alpha
,\beta )}\right) \equiv \left( rI\otimes U_{+},W_{\tau }\otimes I\right),
$$
for some $r>0$. \ Recall that
\begin{equation}
\widehat{T}_{1}e_{\mathbf{k}}=\alpha _{\mathbf{k}}\frac{(\alpha _{\mathbf{k}+%
\mathbf{\epsilon }_{1}}^{2}+\beta _{\mathbf{k}+\mathbf{\epsilon }%
_{1}}^{2})^{1/4}}{(\alpha _{\mathbf{k}}^{2}+\beta _{\mathbf{k}}^{2})^{1/4}}%
e_{\mathbf{k}+\mathbf{\epsilon }_{1}}\text{ and }\widehat{T}_{2}e_{\mathbf{k}%
}=\beta _{\mathbf{k}}\frac{(\alpha _{\mathbf{k}+\mathbf{\epsilon }%
_{2}}^{2}+\beta _{\mathbf{k}+\mathbf{\epsilon }_{2}}^{2})^{1/4}}{(\alpha _{%
\mathbf{k}}^{2}+\beta _{\mathbf{k}}^{2})^{1/4}}e_{\mathbf{k}+\mathbf{%
\epsilon }_{2}}\text{.}  \label{Weight of Polar1}
\end{equation}

By (\ref{Weight of Polar1}), we obtain
\begin{equation*}
c\left( \widehat{W}_{(\alpha ,\beta )}\right) \equiv \left( rI\otimes
U_{+},W_{\tau }\otimes I\right) \circ \left( I\otimes U_{+},W_{\rho}\otimes
I\right) \text{,}
\end{equation*}%
where $W_{\rho}\equiv \mathrm{shift}\left( \left( \frac{r^{2}+\tau _{1}^{2}}{%
r^{2}+\tau _{0}^{2}}\right) ^{\frac{1}{4}},\left( \frac{r^{2}+\tau _{2}^{2}}{%
r^{2}+\tau _{1}^{2}}\right) ^{\frac{1}{4}},\left( \frac{r^{2}+\tau _{3}^{2}}{%
r^{2}+\tau _{2}^{2}}\right) ^{\frac{1}{4}},\cdots \right) $ with $\left\Vert
W_{r}\right\Vert =1$ and $\circ $ denotes Schur product. Note that%
\begin{equation*}
\begin{tabular}{l}
$\mathrm{shift}\left( \widehat{\alpha }_{\left( 0,0\right) },\widehat{\alpha
}_{\left( 1,0\right) },\cdots \right) $ \\
$=\mathrm{shift}\left( \alpha _{\left( 0,0\right) },\alpha _{\left(
1,0\right) },\cdots \right) \circ \mathrm{shift}\left( \left( \frac{\alpha
_{\left( 1,0\right) }^{2}+\beta _{\left( 1,0\right) }^{2}}{\alpha _{\left(
0,0\right) }^{2}+\beta _{\left( 0,0\right) }^{2}}\right) ^{\frac{1}{4}%
},\left( \frac{\alpha _{\left( 2,0\right) }^{2}+\beta _{\left( 2,0\right)
}^{2}}{\alpha _{\left( 1,0\right) }^{2}+\beta _{\left( 1,0\right) }^{2}}%
\right) ^{\frac{1}{4}},\cdots \right) $%
\end{tabular}%
\end{equation*}%
and%
\begin{equation*}
\begin{tabular}{l}
$\mathrm{shift}\left( \widehat{\beta }_{\left( 0,0\right) },\widehat{\beta }%
_{\left( 1,0\right) },\cdots \right) $ \\
$=\mathrm{shift}\left( \beta _{\left( 0,0\right) },\beta _{\left( 1,0\right)
},\cdots \right) \circ \mathrm{shift}\left( \left( \frac{\alpha _{\left(
0,1\right) }^{2}+\beta _{\left( 0,1\right) }^{2}}{\alpha _{\left( 0,0\right)
}^{2}+\beta _{\left( 0,0\right) }^{2}}\right) ^{\frac{1}{4}},\left( \frac{%
\alpha _{\left( 0,2\right) }^{2}+\beta _{\left( 0,2\right) }^{2}}{\alpha
_{\left( 0,1\right) }^{2}+\beta _{\left( 0,1\right) }^{2}}\right) ^{\frac{1}{%
4}},\cdots \right) .$%
\end{tabular}%
\end{equation*}%
Since
\begin{equation*}
\begin{tabular}{l}
$\left\Vert \mathrm{shift}\left( \left( \frac{\alpha _{\left( 1,0\right)
}^{2}+\beta _{\left( 1,0\right) }^{2}}{\alpha _{\left( 0,0\right)
}^{2}+\beta _{\left( 0,0\right) }^{2}}\right) ^{\frac{1}{4}},\left( \frac{%
\alpha _{\left( 2,0\right) }^{2}+\beta _{\left( 2,0\right) }^{2}}{\alpha
_{\left( 1,0\right) }^{2}+\beta _{\left( 1,0\right) }^{2}}\right) ^{\frac{1}{%
4}},\left( \frac{\alpha _{\left( 3,0\right) }^{2}+\beta _{\left( 3,0\right)
}^{2}}{\alpha _{\left( 2,0\right) }^{2}+\beta _{\left( 2,0\right) }^{2}}%
\right) ^{\frac{1}{4}},\cdots \right) \right\Vert $ \\
$=\left\Vert \mathrm{shift}\left( \left( \frac{\alpha _{\left( 0,1\right)
}^{2}+\beta _{\left( 0,1\right) }^{2}}{\alpha _{\left( 0,0\right)
}^{2}+\beta _{\left( 0,0\right) }^{2}}\right) ^{\frac{1}{4}},\left( \frac{%
\alpha _{\left( 0,2\right) }^{2}+\beta _{\left( 0,2\right) }^{2}}{\alpha
_{\left( 0,1\right) }^{2}+\beta _{\left( 0,1\right) }^{2}}\right) ^{\frac{1}{%
4}},\left( \frac{\alpha _{\left( 0,3\right) }^{2}+\beta _{\left( 0,3\right)
}^{2}}{\alpha _{\left( 0,2\right) }^{2}+\beta _{\left( 0,2\right) }^{2}}%
\right) ^{\frac{1}{4}},\cdots \right) \right\Vert =1$,%
\end{tabular}%
\end{equation*}%
we have%
\begin{equation} \label{eq199}
\left\Vert \mathrm{shift}\left( \widehat{\alpha }_{\left( 0,0\right) },%
\widehat{\alpha }_{\left( 1,0\right) },\widehat{\alpha }_{\left( 2,0\right)
},\cdots \right) \right\Vert =\left\Vert \mathrm{shift}\left( \alpha
_{\left( 0,0\right) },\alpha _{\left( 1,0\right) },\alpha _{\left(
2,0\right) },\cdots \right) \right\Vert
\end{equation}%
and%
\begin{equation} \label{eq200}
\left\Vert \mathrm{shift}\left( \widehat{\beta }_{\left( 0,0\right) },%
\widehat{\beta }_{\left( 1,0\right) },\widehat{\beta }_{\left( 2,0\right)
},\cdots \right) \right\Vert =\left\Vert \mathrm{shift}\left( \beta _{\left(
0,0\right) },\beta _{\left( 0,1\right) },\beta _{\left( 0,2\right) },\cdots
\right) \right\Vert \text{.}
\end{equation}%
Thus, by the method used in the Proof of Theorem 8.2 we have
$$
\sigma _{T}\left(\widehat{W}_{(\alpha ,\beta )}\right) =\sigma _{T}\left(W_{(\alpha ,\beta )}\right). 
$$
Observe now that (\ref{eq199}) and (\ref{eq200}) also show that the weighted shifts associated with $0$-th row and the $0$-th column of $\widehat{W}_{(\alpha ,\beta )}$ are compact perturbations of the corresponding weighted shifts for $W_{(\alpha ,\beta )}$. \ As a result, it is straightforward to conclude that   
$$
\sigma _{T_{e}}\left(\widehat{W}_{(\alpha
,\beta )}\right) =\sigma _{T_{e}}\left(W_{(\alpha ,\beta)}\right).
$$
This completes the proof of the theorem.
\end{proof}

In view of Corollary \ref{Cor1}, Remark \ref{Re 3} and Theorem \ref{thm87}, it is natural to formulate the following

\begin{conjecture}
Let $W_{(\alpha ,\beta )}$ be a commuting $2$-variable weighted shift, whose toral and spherical Aluthge transforms are also commuting. \ Then $W_{(\alpha ,\beta )}$, $\widetilde{W}%
_{(\alpha ,\beta )}$ and $\widehat{W}%
_{(\alpha ,\beta )}$ all have the same Taylor spectrum and the same Taylor essential spectrum.
\end{conjecture}


\section{Aluthge Transforms of the Drury-Arveson Shift}

In this section we consider the Drury-Arveson $2$-variable weighted shift $DA$, whose weight sequences are given by

\begin{eqnarray}
\alpha_{(k_1,k_2)}&:=&\sqrt{\frac{k_1+1}{k_1+k_2+1}} \; \; (\textrm{for all} \; k_1,k_2 \ge 0) \\
\beta_{(k_1,k_2)}&:=&\alpha_{(k_2,k_1)} \; \; (\textrm{for all} \; k_1,k_2 \ge 0).
\end{eqnarray}

If we denote the successive rows of the weight diagram of $DA$ by $R_1, R_2, \cdots $, it is easy to see that $R_1=A_1$, the (unweighted) unilateral shift, $R_2=A_2$, the Bergman shift and, more generally, $R_j=A_j$, the Agler $j$-th shift ($j \ge 2$); in particular, all rows and columns are subnormal weighted shifts. \ For $j \ge 2$, the Berger measure of $A_j$ is $(j-1)s^{j-2}ds$ on the closed interval $[0,1]$, and therefore all the Berger measures associated with rows $1, 2, 3, \cdots$ are mutually absolutely continuous, a necessary condition for the subnormality of $DA$. \ However, the Berger measure of the first row ($ds$ on $[0,1]$) is not absolutely continuous with respect to $\delta_1$ (which is the Berger measure of $U_+$, the zeroth-row), and therefore $DA$ cannot be subnormal (by \cite{CuYo2}). \ In fact, a stronger result is true: $DA$ is not jointly hyponormal, as a simple application of the Six-Point Test at $(0,0)$ reveals.

It is also known that $DA \equiv(T_1,T_2)$ is essentially normal; in fact, the commutators $[T_j^*,T_i] \; (i,j=1,2)$ are in the Schatten p-class for $p>2$, as shown by W.A. Arveson \cite{Arv}. \ In the sequel, we prove compactness of the commutators using the homogeneous decomposition of $\ell^2(\mathbb{Z}_+^2)$; this will eventually help us prove that the Aluthge transforms of $DA$ are compact perturbations of $DA$. \ Let $\mathcal{P}_n$ denote the finite dimensional vector space generated by the orthonormal basis vectors $e_{(n,0)},e_{(n-1,1)},\cdots,e_{(0,n)}$, it is easy to see that $\mathcal{P}_n$ is invariant under the action of the self-commutators $[T_i^*,T_i]$ and the cross-commutators $[T_j^*,T_i]$ ($i,j=1,2$). \ A simple calculation reveals that 
$$
[T_1^*,T_1]e_{(0,k_2)}=\frac{1}{k_2+1}e_{(0,k_2)}
$$
$$
[T_1^*,T_1]e_{(k_1,k_2)}=\frac{k_2}{(k_1+k_2))(k_1+k_2+1)}e_{(k_1,k_2)},
$$
so that in $\mathcal{P}_n$ we have
$$
\left\|[T_1^*,T_1]e_{(k_1,n-k_1)}\right\|=\frac{n-k_1}{n(n+1)}.
$$
It follows that the norm of $[T_1^*,T_1]$ restricted to $\mathcal{P}_n$ is bounded by $\frac{1}{n+1}$. \ Since $[T_1^*,T_1]$ is unitarily equivalent to the orthogonal direct sum of its restrictions to the subspaces $\mathcal{P}_n$, we easily conclude that $[T_1^*,T_1]$ is compact. \ The calculation for $[T_2^*,T_2]$ is identical. \ 

In terms of $[T_2^*,T_1]$, one again computes first the action on a generic basis vector in $\mathcal{P}_n$, that is,
$$
[T_2^*,T_1]e_{(n,0)}=0
$$
$$
[T_2^*,T_1]e_{(k_1,n-k_1)}=-\frac{1}{n(n+1)}\sqrt{(k_1+1)(n-k_1)}e_{(k_1+1,n-k_1-1)} \; \; \; (k_1 \ge 1).
$$
It follows that
$$
\left\|[T_2^*,T_1]e_{(k_1,n-k_1)}\right\| \le \frac{1}{n(n+1)}\sqrt{(k_1+1)(n-k_1)} \le \frac{1}{2n} \; \; \; (0 \le k_1 \le n).
$$
As before, $[T_2^*,T_1]$ is an orthogonal direct sum of its restrictions to the subspaces $\mathcal{P}_n$, so the previous estimate proves that $[T_2^*,T_1]$ is compact. \ As a result, we know that $DA$ is essentially normal.

We will now study how much the Aluthge transforms of $DA$ differ from $DA$.

\begin{theorem}
(i) \ $\widetilde{DA}$ is a compact perturbation of $DA$. \newline
(ii) \ $\widehat{DA}$ is a compact perturbation of $DA$.
\end{theorem}

\begin{proof}
(i) \ We first note that the weight sequences $\alpha$ and $\beta$ of $DA$ satisfy (\ref{commuting1}); that is, $\widetilde{DA}$ is commuting. \ Next, we observe that $\widetilde{DA}$ maps $\mathcal{P}_n$ into $\mathcal{P}_{n+1}$ (cf. Lemma \ref{CartAlu}), just as $DA$ does. \ As a result the compactness of $DA-\widetilde{DA}$ will be established once we prove that $\left\|(DA-\widetilde{DA})|_{\mathcal{P}_n}\right\|$ tends to zero as $n \rightarrow \infty$. \ Toward this end, we calculate 
$$
(T_1-\widetilde{T}_1)e_{(k_1,n-k_1)}=(\alpha_{(k_1,n-k_1)}-\sqrt{\alpha_{(k_1,n-k_1)} \alpha_{(k_1+1,n-k_1)}})e_{(k_1+1,n-k_1)}.
$$
Without loss of generality, we focus instead on the expression
$$
\Delta_{\textrm{toral}}:=(\alpha_{(k_1,n-k_1)})^4-(\sqrt{\alpha_{(k_1,n-k_1)} \alpha_{(k_1+1,n-k_1)}})^4;
$$
With the aid of \textit{Mathematica} \cite{Wol}, we obtain
\begin{eqnarray*}
\left|\Delta_{\textrm{toral}}\right| &=& \left|(\alpha_{(k_1,n-k_1)})^4-(\sqrt{\alpha_{(k_1,n-k_1)} \alpha_{(k_1+1,n-k_1)}})^4\right| \\
&=& \frac{(k_1+1)(n-k_1)}{(n+1)^2(n+2)} \le \frac{1}{4(n+2)}.
\end{eqnarray*}
Thus, $\lim_{n \rightarrow \infty} \left\|(DA-\widetilde{DA})|_{\mathcal{P}_n} \right\| = 0$, and therefore $DA-\widetilde{DA}$ is compact.

(ii) As in (i) above, it suffices to prove that $\left\|(DA-\widehat{DA})|_{\mathcal{P}_n}\right\|$ tends to zero as $n \rightarrow \infty$. \ Since 
\begin{eqnarray*}
&&(T_1-\widehat{T}_1)e_{(k_1,n-k_1)} \\
&=&(\alpha_{(k_1,n-k_1)}-\alpha_{(k_1,n-k_1)} (\frac{\alpha_{(k_1+1, n-k_1)}^2 + \beta_{(k_1+1, n-k_1)}^2}
{\alpha_{(k_1, n-k_1)}^2 + \beta_{(k_1, n-k_1)}^2})^{1/4})e_{(k_1+1,n-k_1)},
\end{eqnarray*}
we can again focus on the expression
$$
\Delta_{\textrm{spherical}}:=(\alpha_{(k_1,n-k_1)})^4-(\alpha_{(k_1,n-k_1)} (\frac{\alpha_{(k_1+1, n-k_1)}^2 + \beta_{(k_1+1, n-k_1)}^2}
{\alpha_{(k_1, n-k_1)}^2 + \beta_{(k_1, n-k_1)}^2})^{1/4})^4.
$$
A computation using \textit{Mathematica} \cite{Wol} shows that
$$
\left|\Delta_{\textrm{spherical}}\right|=\frac{k_1+1)(n-k_1)(2n+1)}{n^2(n+1)^2} \le \frac{2n+1}{4n^2}.
$$
We thus conclude, as before, that $DA-\widehat{DA}$ is compact.
\end{proof}

\begin{corollary}
The $2$-variable weighted shifts $DA$, $\widetilde{DA}$ and $\widehat{DA}$ all share the same Taylor spectral picture.
\end{corollary}

\begin{proof}
Since the Taylor essential spectrum and the Fredholm index are invariant under compact perturbations (cf. \cite{Appl}, the result follows from the well know spectral picture of $DA$; that is, $\sigma_T(DA)=\bar{\mathbb{B}^2}$, $\sigma_{Te}(DA)=\partial{\mathbb{B}^2}$, and $\textrm{index} \; DA = \textrm{index} \; \widetilde{DA} = \textrm{index} \; \widehat{DA}$. \ (Here $\mathbb{B}^2$ denotes the unit ball in $\mathbb{C}^2$, and $\partial \mathbb{B}^2$ its topological boundary.)
\end{proof}

\begin{remark}
It is an easy application of the Six-Point Test that neither $\widetilde{DA}$ nor $\widehat{DA}$ is jointly hyponormal.
\end{remark}


\section{Fixed Points of the Spherical Aluthge Transform: \\ Spherically Quasinormal Pairs}\label{Spherquasi}

In this section we discuss the structure of $2$-variable weighted shifts which are fixed points for the spherical Aluthge transform. \ We believe this notion provides the proper generalization of quasinormality to several variables. \ As we noted in the Introduction, a Hilbert space operator $T$ is quasinormal if and only if $T=\widetilde{T}$. \ We use this as our point of departure for the $2$-variable case.

\smallskip

\begin{definition}
A commuting pair $\mathbf{T} \equiv (T_1,T_2)$ is {\it spherically quasinormal} if $\widehat{\mathbf{T}}=\mathbf{T}$.
\end{definition}

We now recall the class of spherically isometric commuting pairs of operators (\cite{Ath1}, \cite{AtPo}, \cite{AtLu}, \cite{EsPu}, \cite{Gle}, \cite{Gle2}).

\begin{definition} \label{spherisom}
A commuting $n$-tuple $\mathbf{T} \equiv (T_1,\cdots,T_n)$ is a spherical isometry if $T_1^*T_1+\cdots+T_n^*T_n=I$.
\end{definition}

In the literature, spherical quasinormality of a commuting $n$-tuple $\mathbf{T} \equiv (T_1,\cdots,T_n)$ is associated with the commutativity of each $T_i$ with $P^2$. \ It is not hard to prove that, for $2$-variable weighted shifts, this is equivalent to requiring that $W_{(\alpha ,\beta )}\equiv (T_1,T_2)$ be a fixed point of the spherical Aluthge transform, that is, $\widehat{W}_{(\alpha ,\beta )}=W_{(\alpha ,\beta )}$. \ A straightforward calculation shows that this is equivalent to requiring that each $U_i$ commutes with $P$. \ In particular, $(U_1,U_2)$ is commuting whenever $(T_1,T_2)$ is commuting. \ Also, recall from Section 1 that a commuting pair $\mathbf{T}$ is a spherical isometry if $P^{2}=I$. \ Thus, in the case of spherically quasinormal $2$-variable weighted shifts, we always have $U_1^*U_1+U_2^*U_2=I$. \ In the following result, the key new ingredient is the equivalence of (i) and (ii). \smallskip

As we noted in the Introduction, the operator $Q:=\sqrt{V_{1}^{\ast }V_{1}+V_{2}^{\ast }V_{2}}$ is a
(joint) partial isometry; for, $PQ^2P=P^2$, from which it follows that $Q$ is isometric on the range of $P$.
In the case when $P$ is injective, we see that a commuting pair $\mathbf{T}\equiv (T_1,T_2) \equiv (V_1P,V_2P)$ is
spherically quasinormal if and only if each $T_{i}$ commutes with $P^2$, and if and only if each $V_{i}$ commutes with $P^2$ ($i=1,2$); in particular, $(V_1,V_2)$ is commuting whenever $(T_1,T_2)$ is commuting. \ Observe also that when $P$ is injective, we always have $V_1^*V_1+V_2^*V_2=I$.

The proof of the following result is a straightforward application of Definition \ref{spherisom}.

\begin{lemma}
A $2$-variable weighted shift $W_{(\alpha,\beta)}$ is a spherical isometry if and only if
$$
\alpha_{\bf{k}}^2+\beta_{\bf{k}}^2=1
$$
for all $\bf{k} \in \mathbb{Z}_+^2$.
\end{lemma}

\begin{lemma} \label{Quasinormal3}
A 2-variable weighted shift $\mathbf{T}$ is spherically quasinormal if and only if there exists $C>0$ such that $\frac{1}{C}\mathbf{T}$ is a spherical isometry, that is, $T_1^*T_1+T_2^*T_2=I$.
\end{lemma}

\begin{proof}
Assume that $\mathbf{T}\equiv (T_{1},T_{2})$ is commuting and spherically quasinormal. \ Then $T_{1}$ and $T_{2}$ commute with $P$. \ We now
consider the following

\noindent \textbf{Claim: \ }For all $\mathbf{k}\equiv (k_{1},k_{2})\in \mathbb{Z}%
_{+}^{2}$, $\alpha _{k}^{2}+\beta _{k}^{2}$ is constant.

\noindent For the proof of Claim, if we fix an orthonormal basis vector $e_{%
\mathbf{k}}$, then \[T_{1}e_{\mathbf{k}}=\alpha _{\mathbf{k}}e_{\mathbf{k+\varepsilon }_{1}}\text{
and }T_{2}e_{\mathbf{k}}:=\beta _{\mathbf{k}}e_{\mathbf{k+\varepsilon }_{2}} \text{,} \] where \textbf{$\varepsilon $}$_{1}:=(1,0)$ and \textbf{$\varepsilon $}$_{2}:=(0,1)$. \ We thus obtain%
\begin{eqnarray*}
PT_{1}e_{\mathbf{k}} &=&\alpha _{(k_{1},k_{2})}\sqrt{\alpha
_{(k_{1}+1,k_{2})}^{2}+\beta _{(k_{1}+1,k_{2})}^{2}}\text{ and} \\
T_{1}Pe_{\mathbf{k}} &=&\sqrt{\alpha _{(k_{1},k_{2})}^{2}+\beta
_{(k_{1},k_{2})}^{2}}\alpha _{(k_{1},k_{2})}\text{.}
\end{eqnarray*}%
It follows that
\begin{equation}
\sqrt{\alpha _{(k_{1}+1,k_{2})}^{2}+\beta _{(k_{1}+1,k_{2})}^{2}}=\sqrt{%
\alpha _{(k_{1},k_{2})}^{2}+\beta _{(k_{1},k_{2})}^{2}}\text{.}
\label{equa1}
\end{equation}%
We also have%
\begin{eqnarray*}
PT_{2}e_{\mathbf{k}} &=&\beta _{(k_{1},k_{2})}\sqrt{\alpha
_{(k_{1},k_{2}+1)}^{2}+\beta _{(k_{1},k_{2}+1)}^{2}}\text{ and } \\
T_{2}Pe_{\mathbf{k}} &=&\sqrt{\alpha _{(k_{1},k_{2})}^{2}+\beta
_{(k_{1},k_{2})}^{2}}\beta _{(k_{1},k_{2})}\text{.}
\end{eqnarray*}%
Hence, we have%
\begin{equation}
\sqrt{\alpha _{(k_{1},k_{2}+1)}^{2}+\beta _{(k_{1},k_{2}+1)}^{2}}=\sqrt{%
\alpha _{(k_{1},k_{2})}^{2}+\beta _{(k_{1},k_{2})}^{2}}\text{.}
\label{equa2}
\end{equation}%
Therefore, by (\ref{equa1}) and (\ref{equa2}), for all $\mathbf{k}\equiv
(k_{1},k_{2})\in \mathbb{Z}_{+}^{2}$, we obtain%
\begin{equation}
\sqrt{\alpha _{(k_{1},k_{2})}^{2}+\beta _{(k_{1},k_{2})}^{2}}=\sqrt{\alpha
_{(k_{1}+1,k_{2})}^{2}+\beta _{(k_{1}+1,k_{2})}^{2}}=\sqrt{\alpha
_{(k_{1},k_{2}+1)}^{2}+\beta _{(k_{1},k_{2}+1)}^{2}}\text{.}  \label{equa3}
\end{equation}
We have thus established the Claim. \ 

\medskip
It follows that $C:=\sqrt{\alpha _{(k_{1},k_{2})}^{2}+\beta _{(k_{1},k_{2})}^{2}}$ is independent of $\mathbf{k}$. \ As a result, $\frac{1}{C}\mathbf{T}$ is a spherical isometry, as desired.
\end{proof}

By the proof of Lemma \ref{Quasinormal3}, we remark that once the zero-th
row of $T_{1}$, call it $W_{0}$, is given, then the entire $2$-variable
weighted shift is fully determined. \ We shall return to this in Subsection \ref{last}.

We now recall a result of A. Athavale.

\begin{theorem} (\cite{Ath1}) \ A spherical isometry is always subnormal.
\end{theorem}

By Lemma \ref{Quasinormal3}, we immediately obtain

\begin{corollary} \label{Qua-sub}
A spherically quasinormal $2$-variable weighted shift is subnormal.
\end{corollary}

We mention in passing two more significant features of spherical isometries.

\begin{theorem} (\cite{MuPt}) \ Spherical isometries are hyporeflexive.
\end{theorem}

\begin{theorem} (\cite{EsPu}) \ For every $n \ge 3$ there exists a non-normal spherical isometry $\mathbf{T}$ such that the polynomially convex hull of $\sigma_T(\mathbf{T})$ is contained in the unit sphere.
\end{theorem}

\begin{remark}
(i) \ A. Athavale and S. Poddar have recently proved that a commuting spherically quasinormal pair is always subnormal \cite[Proposition2.1]{AtPo}; this provides a different proof of Corollary \ref{Qua-sub}. \newline 
(ii) \ In a different direction, let $Q_{\mathbf{T}}(X):=T_1^*XT_1+T_2^*XT_2$. \ By induction, it is easy to prove that if $\mathbf{T}$ is spherically quasinormal, then $Q_{\mathbf{T}}^n(I)=(Q_{\mathbf{T}}(I))^n \; (n \ge 0)$; by \cite[Remark 4.6]{ChSh}, $\mathbf{T}$ is subnormal. 
\end{remark}


\subsection{Construction of spherical isometries} \label{last}

In the class of $2$-variable weighted shifts, there is a simple description of spherical isometries, in terms of the weight sequences $\alpha$ and $\beta$, which we now present. \ Since spherical isometries are (jointly) subnormal, we know that the zeroth-row must be subnormal. \ Start then with a subnormal unilateral weighted shift, and denote its weights by $(\alpha_{(k,0)})_{k=0,1,2,\cdots}$. \ Using the identity 
\begin{equation} \label{sphericalidentity}
\alpha_{\mathbf{k}}^2+\beta_{\mathbf{k}}^2=1 \; \; \; (\mathbf{k} \in \mathbb{Z}_+^2),
\end{equation}
and the above mentioned zeroth-row, we can compute $\beta_{(k,0)}$ for $k=0,1,2,\cdots$. \ With these new values available, we can use the commutativity property (\ref{commuting}) to generate the values of $\alpha$ in the first row (see Figure \ref{Figure 1}); that is, 
$$
\alpha_{(k,1)}:=\alpha_{(k,0)}\beta_{(k+1,0)}/\beta_{(k,0)}.
$$
We can now repeat the algorithm, and calculate the weights $\beta_{(k,1)}$ for $k=0,1,2,\cdots$, again using the identity (\ref{sphericalidentity}). \ This is turn leads to the $\alpha$ weights for the second row, and so on.

This simple construction of spherically isometric $2$-variable weighted shifts will allow us to study properties like recursiveness (tied to the existence of finitely atomic Berger measures) and propagation of recursive relations. \ We pursue these ideas in an upcoming manuscript.

\bigskip

\textit{Acknowledgments.} \ Preliminary versions of some of the results in this paper have been announced in \cite{CR3}. \ Some of the calculations in this paper were obtained using the software tool \textit{Mathematica} \cite{Wol}. 


\end{document}